\documentclass[11pt,reqno]{amsart}

\usepackage{amsfonts}
\usepackage{amsmath,amscd}
\usepackage{amssymb}
\usepackage{tikz-cd}
\usepackage{amsthm}
\usepackage{cases}

\usepackage{galois}
\usepackage{chemarrow}
\usepackage[all]{xy}
\usepackage{graphicx}
\usepackage[colorlinks,
            linkcolor=blue,
            linktoc=all,
            anchorcolor=black,
            citecolor=red
            ]{hyperref}	
\usepackage{mathrsfs}
\usepackage{extarrows}
\usepackage{texnames}
\usepackage[nameinlink,capitalize]{cleveref}

\def\rank{{\textrm{rank}}}

\def\Aut{{\rm Aut}}

\def\End{{\rm End}}

\def\Ker{{\rm Ker}}

\def\Pic{{\rm Pic}}

\def\rank{{\rm rank}}

\def\Tr{{\rm Tr}}

\theoremstyle{plain}
\newtheorem{introtheorem}{Theorem}

\newtheorem{theorem}{Theorem}[section]

\newtheorem{proposition/example}[theorem]{Proposition/Example}
\newtheorem{proposition}[theorem]{Proposition}
\newtheorem{corollary}[theorem]{Corollary}
\newtheorem{lemma}[theorem]{Lemma}

\theoremstyle{definition}
\newtheorem{definition}[theorem]{Definition}
\newtheorem{remark}[theorem]{Remark}
\newtheorem{application}[theorem]{Application}

\newtheorem{conjecture/question}[theorem]{Conjecture/Question}

\newtheorem{remark/definition}[theorem]{Remark/Definition}
\newtheorem{definition/notation}[theorem]{Definition/Notation}

\numberwithin{equation}{section}

\begin{document}
\title{\textbf{Generalized Deligne--Hitchin Twistor Spaces: Construction and Properties}}

\author{Zhi Hu}

\address{ \textsc{School of Mathematics and Statistics, Nanjing University of Science and Technology, Nanjing 210094, China}}

\email{halfask@mail.ustc.edu.cn}

\author{Pengfei Huang}

\address{ Current address: \textsc{Max Planck Institute for Mathematics in the Sciences, Inselstra\ss e 22, 04103 Leipzig, Germany}}

\email{pfhwangmath@gmail.com}

\author{Runhong Zong}

\address{ \textsc{School of Mathematics, Nanjing University, Nanjing  210093, China}}
\email{rzong@nju.edu.cn}


\date{}

\begin{abstract}
In this paper, we generalize the construction of Deligne--Hitchin twistor space by gluing two certain Hodge moduli spaces. We investigate some properties of such generalized Deligne--Hitchin twistor space as a complex analytic manifold. More precisely, we show it admits holomorphic sections whose normal bundle contains a semistable subbundle with positive degree and whose energy is semi-negative, and it carries a balanced metric. Moreover, we also study the automorphism groups of the Hodge moduli spaces and the generalized Deligne--Hitchin twistor spaces.
\end{abstract}

\maketitle

\tableofcontents
\section{Introduction}
The twistor theory originally comes from Penrose's non-linear graviton construction \cite{Pen}, and later was generalized to hyperK\"ahler manifolds by Hitchin, Karlhede, Lindstr\"{o}m and Ro\v{c}ek \cite{HKLR}. Given a hyperK\"ahler manifold $M$ together with three complex structures $I,J,K:=IJ=-JI$, we have an associated complex manifold $\mathrm{TW}(M)$, called the \emph{twistor space},   which encodes the hyperK\"ahler structure on $M$.
 Topologically it is just the product of $M$ with the 2-sphere $\mathbf{S}^2=\{\overrightarrow{x}=(x_1,x_2,x_3)\in \mathbb{R}^3:x_1^2+x_2^2+x_3^2=1\}$, and it has a tautological complex structure given by the natural complex structure on  $\mathbf{S}^2$ and the complex structure $I_{\overrightarrow{x}}=x_1I+x_2J+x_3K$ on $M\times\{\overrightarrow{x}\}$. Hitchin et. al showed that $\mathrm{TW}(M)$ has the following properties
 \begin{itemize}
   \item  $\mathrm{TW}(M)$ is a holomorphic fibration over $\mathbf{S}^2\simeq\mathbb{P}^1$,
   \item $\mathrm{TW}(M)$ has a real structure $\sigma:\mathrm{TW}(M)\rightarrow\mathrm{TW}(M)$ covering the antipodal map of $\mathbb{P}^1$,
   \item the fibration $\pi:\mathrm{TW}(M)\rightarrow\mathbb{P}^1$ admits a family of global holomorphic sections called the \emph{twistor lines}, these sections have normal bundles holomorphically isomorphic to $\mathcal{O}_{\mathbb{P}^1}(1)^{\oplus\dim_{\mathbb{C}}M}$ (\emph{weight-one property}), and they are real with respect to $\sigma$,
   \item there is a holomorphic section $\omega\in H^0(\mathrm{TW}(M),\Lambda^2\Ker(d\pi)^*\otimes\pi^*\mathcal{O}_{\mathbb{P}^1}(2))$ which induces a holomorphic symplectic form on each fiber $\pi^{-1}(\lambda)$.
 \end{itemize}

A highly non-trivial example of the non-compact hyperK\"ahler manifold is the smooth locus $\mathcal{M}_{\mathrm{sol}}$ of the moduli space of solutions to Hitchin's equations over a smooth compact Riemann surface $X$ (of genus $g\geq2$) \cite{NH}. Endowed with one complex structure it becomes a complex analytic manifold as the moduli space $\mathcal{M}_{\mathrm{Dol}}(X,r)$ (called the \emph{Dolbeault moduli space}) of stable Higgs bundles over $X$ of rank $r$ with the vanishing Chern class, also with another complex structure it is made into a complex analytic manifold as the moduli space $\mathcal{M}_{\mathrm{dR}}(X,r)$ (called the \emph{de Rham moduli space}) of irreducible flat bundles over $X$ of rank $r$. When applying Hitchin's twistor construction for $\mathcal{M}_{\mathrm{sol}}$, we can obtain the twistor space $\mathrm{TW}(\mathcal{M}_{\mathrm{sol}})$.

To give a moduli interpretation of $\mathrm{TW}(\mathcal{M}_{\mathrm{sol}})$, Deligne introduced the notion of $\lambda$-flat bundles \cite{D}, as the interpolation of usual flat bundles and Higgs bundles. This idea was later illustrated and developed by Simpson via his series work \cite{CS7,CS8,CS9,CS10,CS11}. This interpretation can be treated as an elegant application of nonabelian Hodge theory.
More precisely, there is a moduli space $\mathcal{M}_{\mathrm{Hod}}(X,r)$ (called the \emph{Hodge moduli space}) that parametrizes the isomorphism classes of stable $\lambda$-flat bundles over $X$ of rank $r$ with the vanishing  Chern class, it has a fibration over the complex line $\mathbb{C}$ such that the fibers over 0 and 1 are $\mathcal{M}_{\mathrm{Dol}}(X,r)$ and $\mathcal{M}_{\mathrm{dR}}(X,r)$, respectively \cite{CS7,CS8,CS9}. Let $\bar{X}$ be the complex conjugate variety of $X$, then the two Hodge moduli spaces $\mathcal{M}_{\mathrm{Hod}}(X,r)$ and $\mathcal{M}_{\mathrm{Hod}}(\bar{X},r)$ can be glued together through the \emph{Deligne isomorphism}, which is defined via Riemann--Hilbert correspondence. This produces a complex analytic manifold $\mathrm{TW}_{\mathrm{DH}}(X,r)$ as a holomorphic fibration over the Riemann sphere $\mathbb{P}^1$. Simpson showed that $\mathrm{TW}_{\mathrm{DH}}(X,r)$ is analytically isomorphic to $\mathrm{TW}(\mathcal{M}_{\mathrm{sol}})$ \cite{CS7}, and called it the \emph{Deligne--Hitchin twistor space}. In particular, the twistor lines are interpreted as the \emph{preferred sections} via Simpson correspondence.

The theory of $\lambda$-flat bundle and
Deligne--Hitchin twistor space attracts a lot of attention and has a vast development after Deligne--Simpson's work. For example, Mochizuki proved the Kobayashi--Hitchin correspondence for (poly-)stable $\lambda$-flat bundles \cite{TM}. In \cite{BGHL}, Biswas, G\'{o}mez,  Hoffman and  Logares proved the Torelli theorem for the Deligne--Hitchin twistor space $\mathrm{TW}_{\mathrm{DH}}(X,r,\mathcal{O})$ related to the objects with trivial determinant, and later this result was generalized to the case of principal objects with fixed topological type in \cite{BGH}. The holomorphic automorphism group $\Aut(\mathrm{TW}_{\mathrm{DH}}(X,1))$ of rank one Deligne--Hitchin twistor space $\mathrm{TW}_{\mathrm{DH}}(X,1)$ was studied by Biswas and Heller in \cite{BH}, and the connected component $\Aut_0(\mathrm{TW}_{\mathrm{DH}}(X,1))$ that contains the identity automorphism was calculated. Biswas, Heller and R\"oser investigated the real holomorphic sections of $\pi:\mathrm{TW}_{\mathrm{DH}}(X,2,\mathcal{O})\rightarrow\mathbb{P}^1$, in particular, they find some real holomorphic sections that are not preferred sections \cite{BHR}. When $X$ is a quasi-projective curve, Simpson also constructed the  Deligne--Hitchin twistor space for the moduli spaces of parabolic logarithmic flat (Higgs) bundles over $X$ (\cite{CS8} for rank $r=1$, \cite{CS10} for rank $r=2$, \cite{CS11} for any rank $r$).  Recently, some authors also used these notions in the framework of $p$-adic geometry \cite{ag,F,ps}.

In the present paper, we generalize Deligne--Simpson's construction described above.  Roughly speaking, our construction is in fact replacing the fiber $\mathcal{M}_{\mathrm{Dol}}(\bar{X},r)$ of the fibration $\pi: \mathrm{TW}_{\mathrm{DH}}(X,r)\to\mathbb{P}^1$ over the point $\infty\in\mathbb{P}^1$ with the Dolbeault moduli space $\mathcal{M}_{\mathrm{Dol}}(X',r)$ for $X'$ representing a Riemann surface that lies in the Teichm\"{u}ller space $\mathrm{Tei}(\mathcal{X})$ of complex structures on the underlying $C^\infty$-manifold $\mathcal{X}$ of $X$. The resulting holomorphic fibration  $\pi:\mathrm{TW}(X,X';r)\rightarrow \mathbb{P}^1$ is called the \emph{generalized Deligne--Hitchin twistor space}, which is also realized via a suitable holomorphic gluing map that generalizes the Deligne isomorphism. Fixing $X$, the analytic structure of generalized Deligne--Hitchin twistor space depends on the choice of $X'$ in  $\mathrm{Tei}(\mathcal{X})$, which dose not  affect the $C^\infty$-structure. In other words, we obtain a family of non-compact complex analytic manifolds effectively parametrized by $\mathrm{Tei}(\mathcal{X})$.

The notion of  twistor line plays the central role in Hitchin's twistor theory since twistor lines enjoy nice properties (especially weight-one property) and cover the entire twistor space. For the generalized Deligne--Hitchin twistor space, we also introduce the twistor line, but some properties should be given up.
 Unlike usual Deligne--Hitchin twsitor space, for a general Higgs bundle lying in Dolbeault moduli space there may not exist a twistor line through it,  even if  exists,  weight-one property may be broken.
To find the twistor lines in our generalized spaces, we introduce another type of sections, namely \emph{de Rham sections}.
 The construction  heavily depends on Simpson's work on Bialynicki--Birula theory for Hodge moduli spaces \cite{CS9,CW,HH,HP,H}. For a generalized Deligne--Hitchin twistor space,
we show that there is a de Rham section through a given $\mathbb{C}^*$-fixed point lying in Dolbeault moduli space such that this section can be exactly realized as a twistor line (in our sense). In particular, if one picks the $\mathbb{C}^*$-fixed point as a stable vector bundle with zero Higgs field, then the twistor line through it is a positive rational curve with degree $\dim_\mathbb{C}\mathcal{M}_{\mathrm{Hod}}(X,r)$ or $\frac{3}{2}\dim_\mathbb{C}\mathcal{M}_{\mathrm{Hod}}(X,r)$.  Recently, Beck, Heller and R\"oser developed a theory on the energy of the sections of Deligne--Hitchin twistor space in \cite{BHR1}. We apply their theory to de Rham sections, in particular, we show that the energy of de Rham sections is semi-negative.

In the last section, we investigate the automorphism group of the generalized Deligne--Hitchin twistor space. Biswas, Heller and R\"oser  proved that the automorphism group of Deligne--Hitchin twistor space homotopic to the identity maps the fibers to fibers \cite{BHR}. We show that such property is still satisfied for the generalized Deligne--Hitchin twistor space.  We also generalize  Biswas--Heller's work on the automorphism group of rank one Hodge moduli space \cite{BH}  to the higher rank cases.


In summary, the main results of this paper can be concluded as the followings.

\begin{introtheorem}
Let $\mathrm{TW}(X,X';r)$ be the generalized Deligne--Hitchin twistor space, and $\mathrm{TW}(X,X';r,\mathcal{O})$ be the generalized Deligne--Hitchin twistor space related to the objects with trivial determinant.
\begin{enumerate}
\item (\textbf{Torelli Theorem} [Theorems \ref{tor}]) Let $(X,X^\prime)\in \mathfrak{TD}(\mathcal{X})$ and $(Y,Y^\prime)\in \mathfrak{TD}(\mathcal{Y})$ be  two pairs of Riemann surfaces of genus $g\geq3$,  where  $\mathfrak{TD}(\mathcal{X}):=\frac{\mathrm{Tei}(\mathcal{X})\times \mathrm{Tei}(\mathcal{X})}{\mathrm{TMCG}(\mathcal{X})}$ for the twistor mapping class group $\mathrm{TMCG}(\mathcal{X}):=\frac{ \{(f,g)\in \mathrm{Diff}(\mathcal{X})\times \mathrm{Diff}(\mathcal{X}): f^{-1}\circ g\in \mathrm{Diff}_0(\mathcal{X})\}}{\mathrm{Diff}_0(\mathcal{X})\times \mathrm{Diff}_0(\mathcal{X})}$.
If there is  the analytic isomorphism $\mathrm{TW}(X,X';r,\mathcal{O})\simeq\mathrm{TW}(Y,Y';r,\mathcal{O})$ of the generalized Deligne--Hitchin twistor spaces for $r\geq 2$, then either $X\simeq Y,X'\simeq Y'$, or $X\simeq Y', X' \simeq Y$ in the space $\mathfrak{TD}(\mathcal{X})$.
    \item (\textbf{Twistor Lines} [Theorem \ref{37}, Theorem \ref{mmmmm}, Theorem \ref{39}]) Given  a  fixed point lying in Dolbeault moduli space, there is a twistor line $s$ through it in $T^*\mathrm{TW}(X,X^\prime;r)$ such that 
         \begin{itemize}
            \item the normal bundle $N_s$ contains a subbundle $L\simeq\mathcal{O}_{\mathbb{P}^1}(1)^{\oplus\nu}$ for $\nu\geq g$,
            \item the energy of $s$ is semi-negative .
         \end{itemize}
\item (\textbf{Balanced Metric} [Theorem \ref{balance}]) $\mathrm{TW}(X,X';r)$  admits a balanced metric, i.e. the induced Hermitian form $\omega$ satisfies $d\omega^{\dim_{\mathbb{C}}(\mathcal{M}_{\mathrm{B}}(\mathcal{X},r))}=0$.
\item (\textbf{Automorphism Groups} [Theorem \ref{auto}]) Let $\Aut_0(\mathrm{TW}(X,X';r))$ be the identity component of the holomorphic automorphism  group $\Aut(\mathrm{TW}(X,X';r))$ of $\mathrm{TW}(X,X';r)$, then each element of $\Aut_0(\mathrm{TW}(X,X';r))$ maps the fibers of the fibration $\pi:\mathrm{TW}(X,X';r)\rightarrow\mathbb{P}^1$ to fibers. Moreover, this group satisfies the following short exact sequence
$$
\mathrm{Id}\longrightarrow\mathfrak{K}\longrightarrow\Aut_0(\mathrm{TW}(X,X';r))\longrightarrow\mathbb{C}^*\longrightarrow\mathrm{Id},
$$
where each element of $\mathfrak{K}$ preserves the fibers of the fibration  $\pi:\mathrm{TW}(X,X';r)\rightarrow\mathbb{P}^1$.
\end{enumerate}
\end{introtheorem}

\bigskip

\noindent\textbf{Acknowledgements}.
The authors would like to thank Professor Carlos Simpson for the kind help, Professor Peter Gothen for telling them about the paper \cite{BGL} on codimension property of the Dolbeault moduli spaces, and Professor Xin Lv and Professor Ruiran Sun for useful discussions on Sect. 5.  The authors would like to express their deep gratitude to the anonymous referee for pointing out some mistakes in earlier versions. Pengfei Huang would like to thank LJAD, Universit\'{e} C\^{o}te d'Azur and Institut f\"{u}r Mathematik, Universit\"{a}t Heidelberg for their kind hospitalities.

\section{Generalized Deligne--Hitchin Twistor Spaces}

\subsection{Constructions}

Throughout the paper, when there is no particular indication, $X$ is always assumed to be a smooth compact Riemann surface of genus $g\geq 2$, and let $K_X:=\mathrm{\Omega}_X^1$ be the canonical line bundle over $X$.

 The underlying $C^\infty$-manifold of $X$ is denoted by $\mathcal{X}$, and we can write $X=(\mathcal{X}, I)$ with  a chosen complex structure $I$ that determines the Riemann surface structure. The Teichm\"{u}ller space  $\mathrm{Tei}(\mathcal{X})$ of  $\mathcal{X}$  is the space $\mathrm{CS}(\mathcal{X})$ of
complex structures on $\mathcal{X}$  modulo diffeomorphisms homotopy
equivalent to the identity, i.e.
\begin{align*}
 \mathrm{Tei}(\mathcal{X})=\mathrm{CS}(\mathcal{X})/\mathrm{Diff}_0(\mathcal{X}),
\end{align*}
where $\textrm{Diff}_0(\mathcal{X})$ being  the identity component of the diffeomorphism group $\textrm{Diff}(\mathcal{X})$,
and the moduli space of complex structures on
$\mathcal{X}$ is  quotient of the Teichm\"{u}ller space under the action of extended mapping class group $\mathrm{Mod}(\mathcal{X})$, i.e.
\begin{align*}
\mathcal{M}(\mathcal{X})=\mathrm{Tei}(\mathcal{X})/ \mathrm{Mod}(\mathcal{X}),
\end{align*}
where
 \begin{align*}
  \mathrm{Mod}(\mathcal{X}):=\pi_0(\mathrm{Diff}(\mathcal{X}))=\mathrm{Diff}(\mathcal{X})/\mathrm{Diff}_0(\mathcal{X}),
 \end{align*}
is  defined as the group of  isotopy classes of all diffeomorphisms of $\mathcal{X}$.

The fundamental group of $\mathcal{X}$ is given by
$$
\pi_1(\mathcal{X})=\bigg\langle\alpha_1,\beta_1,\cdots,\alpha_g,\beta_g:\prod_{i=1}^g\alpha_i\beta_i\alpha_i^{-1}\beta_i^{-1}=1\bigg\rangle.
$$
 Let  $\mathrm{Inn}(\pi_1(\mathcal{X}))$  be the subgroup  of $\Aut(\pi_1(\mathcal{X}))$ that consists of inner automorphisms of $\pi_1(\mathcal{X})$, and then define the outer automorphism group
\begin{align*}
 \Gamma_\mathcal{X}:=\mathrm{Out}(\pi_1(\mathcal{X}))=\Aut(\pi_1(\mathcal{X}))/\mathrm{Inn}(\pi_1(\mathcal{X}))
\end{align*}
 The beautiful Dehn--Nielsen--Baer theorem  equals the  topologically defined group, $ \mathrm{Mod}(\mathcal{X})$, with an
algebraically defined group, $\Gamma_\mathcal{X}$ \cite{FM2012}.

\hspace*{\fill} \\
 \noindent  \emph{\textbf{Deligne's approach: via generalized Delgine isomorphisms} }
\hspace*{\fill} \\

The space of representations of $\pi_1(\mathcal{X})$ into $G=\mathrm{GL}(r,\mathbb{C})$ is denoted by $R(\mathcal{X},G)$, which contains a subspace $R_{\mathrm{irr}}(\mathcal{X},G)$ consisting of irreducible representations, and then the \emph{Betti moduli space} is defined as
$$
\mathcal{M}_\mathrm{B}(\mathcal{X},r)=R_{\mathrm{irr}}(\mathcal{X},G)/G,
$$
namely the (geometric) quotient space of $R_{\mathrm{irr}}(\mathcal{X},G)$ by the conjugate action of $G$. It's known that $R_{\mathrm{irr}}(\mathcal{X},G)$ is a quasi-affine  variety by the embedding
$$
R_{\mathrm{irr}}(\mathcal{X},G)\hookrightarrow G^{2g},\quad
\rho\mapsto (\rho(\alpha_1),\rho(\beta_1),\cdots,\rho(\alpha_g),\rho(\beta_g)),
$$
and $\mathcal{M}_\mathrm{B}(\mathcal{X},r)$ is a smooth quasi-projective  variety.
Since  $\mathrm{Inn}(\pi_1(\mathcal{X}))$  preserves the $G$-orbits on $R_{\mathrm{irr}}(\mathcal{X},G)$, the outer automorphism group $
 \Gamma_\mathcal{X}$
 acts on $M_\mathrm{B}(\mathcal{X},G)$.

By Riemann--Hilbert correspondence, $\mathcal{M}_\mathrm{B}(\mathcal{X},r)$ is complex analytically isomorphic to the moduli space $ \mathcal{M}_{\mathrm{dR}}(X,r) $ of irreducible flat bundles over $X$ of rank $r$.  For the Riemann surface $X'=(\mathcal{X},I')\in\mathcal{M}(\mathcal{X})$ given by another complex structure $I'$, we obviously have the complex analytic isomorphisms
$$
\mathcal{M}_{\mathrm{dR}}(X',r)\simeq\mathcal{M}_{\mathrm{dR}}(X,r)\simeq\mathcal{M}_\mathrm{B}(\mathcal{X},r).
$$
By nonabelian Hodge correspondence \cite{CS4,CS5,CS6}, $\mathcal{M}_\mathrm{B}(\mathcal{X},r)$ is diffeomorphic to the moduli space $ \mathcal{M}_{\mathrm{Dol}}(X,r) $ of stable Higgs bundles over $X$ of rank $r$ with the vanishing Chern class. As a complex analytic manifold, the complex structure of $ \mathcal{M}_{\mathrm{Dol}}(X,r) $ is inherited from that of $X$.

As an interpolation of flat bundles and Higgs bundles, Deligne introduced the notion of $\lambda$-flat bundles \cite{D}, which is illustrated by Simpson in \cite{CS7} and further studied in \cite{CS8,CS9}. Let's recall some general definitions here.

\begin{definition}[\cite{CS7,TM}]
Let $X$ be a smooth complex projective variety of dimension $n$, and assume $\lambda\in\mathbb{C}$.
\begin{enumerate}
  \item Let $E$ be a holomorphic vector bundle over $X$ with the holomorphic structure $\bar\partial_E$.
A \emph{ $\lambda$-connection} on $E$ is a $\mathbb{C}$-linear map $D^\lambda: E\to E\otimes_{\mathcal{O}_X} \mathrm{\Omega}_X^{1}$\footnote{We also denote $D^1$ by $\nabla$.} that satisfies the following $\lambda$-twisted Leibniz rule:
$$
D^\lambda(fs)=fD^\lambda s+\lambda s\otimes  df,
$$
where $f$ and $s$ are holomorphic sections of $\mathcal{O}_X$ and $E$, respectively. It naturally extends to a map $D^\lambda: E\otimes_{\mathcal{O}_X} \mathrm{\Omega}_X^{p}\to E\otimes_{\mathcal{O}_X} \mathrm{\Omega}_X^{p+1}$ for any integer $p\geq0$. If $D^\lambda\circ D^\lambda=0$, then we call $D^\lambda$ a  \emph{flat $\lambda$-connection} and the pair $(E,D^\lambda)$ is called a  \emph{$\lambda$-flat bundle}.
  \item Let $L$ be a fixed ample line bundle over $X$. A  $\lambda$-flat bundle $(E,D^\lambda)$ over $X$ is called \emph{stable} (resp. \emph{semistable}) if for any $\lambda$-flat coherent subsheaf $(V,D^\lambda|_V)$ of $0<\rank(V)<\rank(E)$, we have the following inequality
$$
\mu_L(V)< (\textrm{resp.} \leq) \,\mu_L(E),
$$
where $\mu_L(\bullet)=\frac{\deg(\bullet)}{\rank(\bullet)}=\frac{\int_Xc_1(\bullet)\wedge c_1(L)^{n-1}}{\rank(\bullet)}$ denotes the \emph{slope} of bundle/sheaf with respect to $L$, and we call $(E,D^\lambda)$ is \emph{polystable} if it decomposes as a direct sum of stable $\lambda$-flat bundles with the same slope.
\end{enumerate}
\end{definition}

Let $\mathbf{M}_{\mathrm{Hod}}(X,r)$ be the moduli stack of
$\lambda$-flat (varying $\lambda$  in $\mathbb{C}$) bundles over $X$ of rank $r$ with the vanishing  Chern class,  and let  $\mathbb{M}_{\mathrm{Hod}}(X,r)$ be the coarse moduli space for the semistable stratum of this stack. It's known that  $\mathbb{M}_{\mathrm{Hod}}(X,r)$ is a quasi-projective variety which parametrizes the isomorphism classes of  polystable $\lambda$-flat bundles \cite{CS7}. Moreover, let $\mathcal{M}_{\mathrm{Hod}}(X,r)$ be the moduli space that parametrizes the isomorphism classes of  stable $\lambda$-flat bundles, then it is smooth and is a Zariski dense open subset of $\mathbb{M}_{\mathrm{Hod}}(X,r)$. Obviously, there is a natural fibration $\pi: \mathcal{M}_{\mathrm{Hod}}(X,r)\rightarrow\mathbb{C}, (E,D^\lambda)\mapsto\lambda$\footnote{Throughout the paper, when there is no ambiguity, we use the same notation for a representative and its equivalence class as a point in certain moduli space, for example, expressions like ``$(E,D^\lambda)\in\mathcal{M}_{\mathrm{Hod}}(X,r)$'' and ``$(E_1,D_1^\lambda)=(E_2,D_2^\lambda)$'' are used in this sense.}.  In particular,  $\pi^{-1}(0)=\mathcal{M}_{\mathrm{Dol}}(X,r)$,  and  $\pi^{-1}(\lambda)=:\mathcal{M}_\lambda(X,r)$  is algebraically isomorphic to $\mathcal{M}_{\mathrm{dR}}(X,r)$ for any $\lambda\in\mathbb{C}^*$.

\begin{definition}
The \emph{twistor datum} on $\mathcal{X}$ is a triple $\mathbf{TD}_\mathcal{X}=(I,I',\gamma)$ consisting of two complex structures $I,I'$ on $\mathcal{X}$
and $\gamma\in \mathrm{Diff}(\mathcal{X})$. Two twistor data $\mathbf{TD}^1_\mathcal{X}=(I_1,I_1',\gamma_1)$ and $\mathbf{TD}^2_\mathcal{X}=(I_2,I_2',\gamma_2)$ on $\mathcal{X}$ is called \emph{equivalent} if there exit diffeomorphisms $f,g\in \mathrm{Diff}(\mathcal{X})$ such that
\begin{itemize}
  \item $f_*\circ I_1\circ (f^{-1})_*=I_2$,
  \item $g_*\circ I'_1\circ (g^{-1})_*=I'_2$
  \item $f^{-1}\circ\gamma_2^{-1}\circ g\circ \gamma_1\in \mathrm{Diff}_0(\mathcal{X})$,
\end{itemize}
where $\textrm{Diff}_0(\mathcal{X})$ denotes to the identity component of $\mathrm{Diff}(\mathcal{X})$.
\end{definition}

Note that the twistor data $\mathbf{TD}^1_\mathcal{X}=(I,I',\gamma)$ is equivalent to twistor data $\mathbf{TD}^2_\mathcal{X}=(I,I'',\mathrm{Id})$ for $I''=(\gamma^{-1})_*\circ I'\circ \gamma_*$, then we can choose the  gauge fixing $\gamma=\mathrm{Id}$, hence the equivalent classes of  twistor data are parameterized by the space 
  \begin{align*}
   \mathfrak{TD}(\mathcal{X}):=\frac{ CS(\mathcal{X})\times  CS(\mathcal{X})}{\mathrm{TW}(\mathcal{X})}=\frac{\mathrm{Tei}(\mathcal{X})\times \mathrm{Tei}(\mathcal{X})}{\mathrm{TMCG}(\mathcal{X})},
  \end{align*}
 where the subgroup $\mathrm{TW}(\mathcal{X})$ of $\mathrm{Diff}(\mathcal{X})\times \mathrm{Diff}(\mathcal{X})$ is defined by
\begin{align*}
  \mathrm{TW}(\mathcal{X})=\{(f,g)\in \mathrm{Diff}(\mathcal{X})\times \mathrm{Diff}(\mathcal{X}): f^{-1}\circ g\in \mathrm{Diff}_0(\mathcal{X})\},
\end{align*}
and the \emph{twistor mapping class group} $\mathrm{TMCG}(\mathcal{X})$ is defined by
\begin{align*}
 \mathrm{TMCG}(\mathcal{X})=\frac{ \mathrm{TW}(\mathcal{X})}{\mathrm{Diff}_0(\mathcal{X})\times \mathrm{Diff}_0(\mathcal{X})}.
\end{align*}

Note that as sets, $\mathrm{TMCG}(\mathcal{X})$ one-to-one corresponds to the extended mapping class group $ \mathrm{Mod}(\mathcal{X})$, but they have different group structures.

Given a twistor data $\mathbf{TD}_\mathcal{X}=(I,I',\gamma)$, we can construct the generalized Deligne--Hitchin twistor space.
By  Riemann--Hilbert correspondence,
one  glues $\mathcal{M}_{\mathrm{Hod}}(X,r)$ and $\mathcal{M}_{\mathrm{Hod}}(X',r)$ along the overlap $\mathcal{M}_{\mathrm{Hod}}(X,r)\times_{\mathbb{C}}\mathbb{C}^*\simeq \mathcal{M}_{\mathrm{Hod}}(X',r)\times_{\mathbb{C}}\mathbb{C}^*\simeq \mathcal{M}_\mathrm{B}(\mathcal{X},r)\times \mathbb{C}^*$. More precisely,   one defines the gluing map
\begin{align*}
 \mathbf{d}_{\mathbf{TD}_\mathcal{X}}:\mathcal{M}_{\mathrm{Hod}}(X,r)\times_{\mathbb{C}}\mathbb{C}^*\rightarrow \mathcal{M}_{\mathrm{Hod}}(X',r)\times_{\mathbb{C}}\mathbb{C}^*,
\end{align*}
 called the \emph{generalized Deligne isomorphism},  by the following composition
\begin{align*}
  \CD
    \mathcal{M}_{\mathrm{Hod}}(X,r)\times_{\mathbb{C}}\mathbb{C}^*  @>\mathbf{d_{\mathbf{TD}_\mathcal{X}}}>>   \mathcal{M}_{\mathrm{Hod}}(X',r)\times_{\mathbb{C}}\mathbb{C}^* \\
    @V \iota_\lambda VV @AA\iota^{-1}_{\lambda^{-1}}A   \\
   \mathcal{M}_{\mathrm{dR}}(X,r)\times\mathbb{C}^*    & &  \mathcal{M}_{\mathrm{dR}}(X',r)\times\mathbb{C}^*\\
    @V \mathrm{RH} VV @AA\mathrm{RH }A  \\
    \mathcal{ M}_{\mathrm{B}}(\mathcal{X},r)\times\mathbb{C}^* @>\gamma_\sharp>>  \mathcal{ M}_{\mathrm{B}}(\mathcal{X},r)\times\mathbb{C}^*
  \endCD,  \\
\end{align*}
where the isomorphisms $\iota_\lambda, \textrm{RH}, \gamma_\sharp$ are   given by, respectively,
\begin{align*}
 \iota_\lambda:&\ \mathcal{M}_{\mathrm{Hod}}(X,r)\times_{\mathbb{C}}\mathbb{C}^*\rightarrow\mathcal{M}_{\mathrm{dR}}(X,r)\times\mathbb{C}^*\\
   &\ (((E,\bar\partial_E),D^\lambda),\lambda)\mapsto(((E,\bar\partial_E),\lambda^{-1}D^\lambda),\lambda);\\
  \mathrm{RH}:&\  \mathcal{M}_{\mathrm{dR}}(X,r)\times\mathbb{C}^*\rightarrow \mathcal{ M}_{\mathrm{B}}(\mathcal{X},r)\times\mathbb{C}^*\\
  & \ (((E,\bar\partial_E),\nabla),\lambda)\mapsto(\rho,\lambda)\\& \ \ \  \textrm{ for } \rho\textrm{ being a representation of } \pi_1(\mathcal{X}) \textrm{ determined by Riemann-Hilbert correpondence} ;\\
  \gamma_\sharp:&\  \mathcal{ M}_{\mathrm{B}}(\mathcal{X},r)\times\mathbb{C}^*\rightarrow \mathcal{ M}_{\mathrm{B}}(\mathcal{X},r)\times\mathbb{C}^*\\
 &\ (\rho,\lambda)\mapsto (\rho\circ (\gamma^{-1})_*,\lambda^{-1}).
\end{align*}
Since  the generalized Deligne isomorphism $\mathbf{d}_{\mathbf{TD}_\mathcal{X}}$ covers the map $\mathbb{C}^*\rightarrow\mathbb{C}^*,\lambda\mapsto \lambda^{-1}$, the resulting analytic manifold, denoted by $\mathrm{TW}(\mathbf{TD}_\mathcal{X};r)$, forms a   fibration $\pi:\mathrm{TW}(X,X';r)\rightarrow\mathbb{P}^1$ with the fibers $\pi^{-1}(\lambda)$ complex analytically isomorphic to $\mathcal{M}_\mathrm{B}(\mathcal{X},r)$ for $\lambda\in \mathbb{P}^1\backslash\{0,\infty\}$ and the fibers $\pi^{-1}(0)=\mathcal{M}_{\mathrm{Dol}}(X,r), \pi^{-1}(\infty)=\mathcal{M}_{\mathrm{Dol}}(X',r)$.
We will call $\mathrm{TW}(\mathbf{TD}_\mathcal{X};r)$ the \emph{generalized Deligne--Hitchin twistor space} associated to the twistor datum $\mathbf{TD}_\mathcal{X}$.

\begin{theorem}\label{23}
If the twistor data $\mathbf{TD}^1_\mathcal{X}=(I_1,I_1',\gamma_1)$ is equivalent to $\mathbf{TD}^2_\mathcal{X}=(I_2,I_2',\gamma_2)$, then the generalized Deligne--Hitchin twistor spaces $\mathrm{TW}(\mathbf{TD}^1_\mathcal{X};r)$ is analytically isomorphic to $\mathrm{TW}(\mathbf{TD}^2_\mathcal{X};r)$.
\end{theorem}

\begin{proof}Denote $X_1=(\mathcal{X},I_1),  X_1'=(\mathcal{X},I_1'), X_2=(\mathcal{X},I_2), X_2'=(\mathcal{X},I_2')$. Since $f_*\circ I_1\circ (f^{-1})_*=I_2, g_*\circ I'_1\circ (g^{-1})_*=I'_2$ for diffeomorphisms $f,g\in \textrm{Diff}(\mathcal{X})$, namely, as  complex manifolds, $X_1$ and $X'_1$ are analytically isomorphic to $X_2$ and $X'_2$, respectively, the  biholomorphic equivalences $f: X_1\rightarrow X_2,g: X_1'\rightarrow X_2'$ induce the isomorphisms between moduli spaces $\mathcal{M}_{\mathrm{Hod}}(X_1,r)$ and $\mathcal{M}_{\mathrm{Hod}}(X_2,r)$, $\mathcal{M}_{\mathrm{Hod}}(X'_1,r)$ and $\mathcal{M}_{\mathrm{Hod}}(X'_2,r)$, respectively,  via pullbacks on $\lambda$-flat bundles. To show $\mathrm{TW}(\mathbf{TD}^1_\mathcal{X};r)$ is analytically isomorphic to $\mathrm{TW}(\mathbf{TD}^2_\mathcal{X};r)$, we only need to check  the commutativity of the  following    diagram
\begin{align}\label{ds}
 \CD
 \mathcal{M}_{\mathrm{Hod}}(X_1,r)\times_\mathbb{C}\mathbb{C}^*   @>\mathbf{d}_{\mathbf{TD}^1_\mathcal{X}}>> \mathcal{M}_{\mathrm{Hod}}(X'_1,r)\times_\mathbb{C}\mathbb{C}^*  \\
   @V (f^{-1})^* VV @V (g^{-1})^* VV  \\
  \mathcal{M}_{\mathrm{Hod}}(X_2,r)\times_\mathbb{C}\mathbb{C}^*   @>\mathbf{d}_{\mathbf{TD}^2_\mathcal{X}}>> \mathcal{M}_{\mathrm{Hod}}(X'_2,r)\times_\mathbb{C}\mathbb{C}^* \endCD.
\end{align}

By definition, we  can explicitly express the generalized Deligne isomorphisms  \begin{align*}
   \mathbf{d}_{\mathbf{TD}^i_\mathcal{X}}:\mathcal{M}_{\mathrm{Hod}}(X_i,r)\times_{\mathbb{C}}\mathbb{C}^*&\rightarrow \mathcal{M}_{\mathrm{Hod}}(X'_i,r)\times_{\mathbb{C}}\mathbb{C}^*\\
  (((E,\bar{\partial}_E),D^\lambda),\lambda)&\mapsto(((E,((\gamma_i^{-1})^*[\lambda^{-1}D^{\lambda}+\bar{\partial}_E])_{X_i'}^{0,1}),\lambda^{-1}((\gamma_i^{-1})^*[\lambda^{-1}D^{\lambda}+\bar{\partial}_E])_{X_i'}^{1,0}),\lambda^{-1}),
 \end{align*}
 where $i=1,2$,  $(\bullet)^{1,0}_{X_i^\prime}$ and $(\bullet)^{0,1}_{X_i'}$ denote the corresponding (1,0)-part and (0,1)-part  with respect to the complex structure $I_i'$ of $X_i'$, respectively. Therefore, we have
 \begin{align*}
   &(g^{-1})^*\circ\mathbf{d}_{\mathbf{TD}^1_\mathcal{X}}(((E,\bar{\partial}_E),D^\lambda),\lambda)\\
   =&\ (((E,(g^{-1})^*(((\gamma_1^{-1})^*[\lambda^{-1}D^{\lambda}+\bar{\partial}_E])_{X_1'}^{0,1})),\lambda^{-1}(g^{-1})^*(((\gamma_1^{-1})^*[\lambda^{-1}D^{\lambda}+\bar{\partial}_E])_{X_1'}^{1,0})),\lambda^{-1}),\\
   &\mathbf{d}_{\mathbf{TD}^2_\mathcal{X}}\circ(g^{-1})^*(((E,\bar{\partial}_E),D^\lambda),\lambda)\\
  =&\  (( (E,((\gamma_2^{-1})^*\circ(f^{-1})^*[\lambda^{-1}D^{\lambda}+\bar{\partial}_E])_{X_2'}^{0,1}),\lambda^{-1}((\gamma_2^{-1})^*\circ(f^{-1})^*[\lambda^{-1}D^{\lambda}+\bar{\partial}_E])_{X_2'}^{1,0}), \lambda^{-1})\\
 =&\  (( (E,(g^{-1})^*((g^*\circ(\gamma_2^{-1})^*\circ(f^{-1})^*[\lambda^{-1}D^{\lambda}+\bar{\partial}_E])_{X_1'}^{0,1})),\\
 &\ \ \ \lambda^{-1}(g^{-1})^*((g^*\circ(\gamma_2^{-1})^*\circ(f^{-1})^*[\lambda^{-1}D^{\lambda}+\bar{\partial}_E])_{X_1'}^{1,0}), \lambda^{-1}).
 \end{align*}
 Since $f^{-1}\circ\gamma_2^{-1}\circ g\circ \gamma_1\in \mathrm{Diff}_0(\mathcal{X})$, namely, the diffeomorphisms $f^{-1}\circ\gamma_2^{-1}\circ g$ and $\gamma_1$ lie in the same connected component of $\mathrm{Diff}(\mathcal{X})$, due to Dehn--Nielsen--Baer theorem,  the induced action of $f^{-1}\circ\gamma_2^{-1}\circ g$ on the moduli space $\mathcal{ M}_{\mathrm{B}}(\mathcal{X},r)$ coincide with the induced action of $\gamma_1^{-1}$, which implies that $(g^{-1})^*\circ\mathbf{d}_{\mathbf{TD}^1_\mathcal{X}}(((E,\bar{\partial}_E),D^\lambda),\lambda)$ and $\mathbf{d}_{\mathbf{TD}^2_\mathcal{X}}\circ(g^{-1})^*(((E,\bar{\partial}_E),D^\lambda),\lambda)$
   represent the same point lying in  $\mathcal{M}_{\mathrm{Hod}}(X'_2,r)\times_\mathbb{C}\mathbb{C}^*$, i.e. the  diagram \eqref{ds} is commutative.
\end{proof}

\begin{remark}\
\begin{enumerate}
               \item Given $r$, the $C^\infty$-structures on generalized Deligne--Hitchin twistor spaces are independent of the choices of twistor data. Indeed, by nonabelian Hodge correspondence, there are $C^\infty$-diffeomorphisms
\begin{align*}
  \mathcal{M}_{\mathrm{Hod}}(X,r)\simeq \mathcal{M}_{\mathrm{Hod}}(X',r)\simeq \mathcal{M}_{\mathrm{B}}(\mathcal{X},r)\times\mathbb{C},
\end{align*}
under which, the generalized Deligne isomorphism identifies $(\rho,\lambda)$ with $(\rho,\lambda^{-1})$. Thus we get a global $C^\infty$-trivialization $\mathrm{TW}(\mathbf{TD}_\mathcal{X};r)\simeq \mathcal{M}_\mathrm{B}(\mathcal{X},r)\times \mathbb{P}^1$.
 \item From now on,  we only need to consider the generalized Deligne--Hitchin twistor space associated to the twistor data $\mathbf{TD}_\mathcal{X}=(I,I',\mathrm{Id})$, which
is also denoted by $\mathrm{TW}(X,X';r)$, and for this case, the generalized Deligne isomorphism is simply denoted by $\mathbf{d}$.
              \end{enumerate}
\end{remark}

\hspace*{\fill} \\
 \noindent\emph{\textbf{Hitchin's approach: via  hypercomplex structures}}
\hspace*{\fill} \\

One identifies  $\mathbb{P}^1$ with the sphere $\mathbf{S}^2=\{\overrightarrow{x}=(x_1,x_2,x_3)\in \mathbb{R}^3:x_1^2+x_2^2+x_3^2=1\}$  via the stereographic projections
 \begin{align*}
     p_\mathbf{N}: \mathbf{N}=\mathbf{S}^2\backslash\{(1,0,0)\}&\longrightarrow \mathbb{C}\\
     (x_1,x_2,x_3)&\mapsto \frac{x_2+\sqrt{-1}x_3}{1-x_1}
     \end{align*}
 with the inverse
 \begin{align*}
   p_\mathbf{N}^{-1}: \mathbb{C}&\longrightarrow \mathbf{N}\\
   \lambda&\mapsto \bigg(\frac{|\lambda|^2-1}{|\lambda|^2+1}, \frac{2\mathrm{Re}\lambda}{|\lambda|^2+1},\frac{2\mathrm{Im}\lambda}{|\lambda|^2+1}\bigg),
 \end{align*}
 and
 \begin{align*}
     p_\mathbf{S}: \mathbf{S}=\mathbf{S}^2\backslash\{(-1,0,0)\}&\longrightarrow \mathbb{C}\\
     (x_1,x_2,x_3)&\mapsto \frac{x_2-\sqrt{-1}x_3}{1+x_1}
     \end{align*}
 with the inverse
 \begin{align*}
   p_\mathbf{S}^{-1}: \mathbb{C}&\longrightarrow \mathbf{S}\\
   \lambda&\mapsto \bigg(\frac{1-|\lambda|^2}{|\lambda|^2+1}, \frac{2\mathrm{Re}\lambda}{|\lambda|^2+1},-\frac{2\mathrm{Im}\lambda}{|\lambda|^2+1}\bigg).
 \end{align*}
 On the other hand, let $\widetilde{\mathcal{M}}_{\mathrm{Dol}}(X,r)$ be the underlying $C^\infty$-manifold of the moduli space $\mathcal{M}_{\mathrm{Dol}}(X,r)$, it is well-known that  $\widetilde{\mathcal{M}}_{\mathrm{Dol}}(X,r)$ carries a hypercomplex (hyperK\"{a}hler) structure consisting of three complex structures $\mathbb{I},\mathbb{J},\mathbb{K}$ satisfy the quaternionic relation $\mathbb{I}\mathbb{J}=-\mathbb{J}\mathbb{I}=\mathbb{K}$. In particular, $\mathbb{I}$ comes from the complex structure of $X$ such that $(\widetilde{\mathcal{M}}_{\mathrm{Dol}}(X,r),\mathbb{I})$ is complex analytically  isomorphic to $\mathcal{M}_{\mathrm{Dol}}(X,r)$, and $\mathbb{J}$ comes from the complex structure of the complex Lie group $G$ such that $(\widetilde{\mathcal{M}}_{\mathrm{Dol}}(X,r),\mathbb{J})$ is complex analytically isomorphic to $\mathcal{M}_{\mathrm{dR}}(X,r)$. Similarly,  $\widetilde{\mathcal{M}}_{\mathrm{Dol}}(X',r)$ (that is diffeomorphic to $\widetilde{\mathcal{M}}_{\mathrm{Dol}}(X,r)$) also admits three complex structures $\mathbb{I}', \mathbb{J}, \mathbb{K}'=\mathbb{I}'\mathbb{J}=-\mathbb{J}\mathbb{I}'$. For the products $\widetilde{\mathcal{M}}_{\mathrm{Dol}}(X,r)\times \mathbf{S}$ and $\widetilde{\mathcal{M}}_{\mathrm{Dol}}(X',r)\times \mathbf{N}$, we assign the fibers $\widetilde{\mathcal{M}}_{\mathrm{Dol}}(X,r)$ over $\overrightarrow{x}=(x_1,x_2,x_3)\in\mathbf{S}$ and $\overrightarrow{x'}=(x_1',x_2',x_3')\in \mathbf{N}$ with the complex structures (cf. {\cite[Section 3]{CS7}})
 \begin{align*}
  \mathbb{I}_{\overrightarrow{x}}&=x_1\mathbb{I}+x_2\mathbb{J}+x_3\mathbb{K}\\
&=[1-(\mathrm{Re}\lambda) \mathbb{K}+(\mathrm{Im}\lambda) \mathbb{J}]^{-1}\mathbb{I}[1-(\mathrm{Re}\lambda) \mathbb{K}+(\mathrm{Im}\lambda) \mathbb{J}],\\
  \mathbb{I}_{\overrightarrow{x'}}'&=-x'_1\mathbb{I}'+x'_2\mathbb{J}-x'_3\mathbb{K}'\\
&=[1-(\mathrm{Re}\lambda') \mathbb{K}'+(\mathrm{Im}\lambda') \mathbb{J}]^{-1}\mathbb{I}'[1-(\mathrm{Re}\lambda') \mathbb{K}'+(\mathrm{Im}\lambda) \mathbb{J}],
 \end{align*}
respectively, where $\lambda= p_\mathbf{S}(\overrightarrow{x}), \lambda'=p_\mathbf{N}(\overrightarrow{x'})$. By Simpson's result ({\cite[Theorem 4.2]{CS7}}), if  $\overrightarrow{x}=(x_1,x_2,x_3)$ and $\overrightarrow{x'}=(-x_1,x_2,-x_3)$ both lie in $\mathbf{S}\bigcap\mathbf{N}$,  there is a diffeomorphism
$f_{\overrightarrow{x},\overrightarrow{x'}}:\widetilde{\mathcal{M}}_{\mathrm{Dol}}(X,r)\rightarrow\widetilde{\mathcal{M}}_{\mathrm{Dol}}(X',r)$ such that it is a biholomorphic equivalence with respect to the complex structures $\mathbb{I}_{\overrightarrow{x}}$ and $\mathbb{I}_{\overrightarrow{x'}}'$. Moreover, by virtue of these diffeomorphisms, we can glue $\widetilde{\mathcal{M}}_{\mathrm{Dol}}(X,r)\times \mathbf{S}$ and $\widetilde{\mathcal{M}}_{\mathrm{Dol}}(X',r)\times \mathbf{N}$ along the overlap $\mathbf{S}\bigcap\mathbf{N}$ to make the topological product $\widetilde{\mathcal{M}}_{\mathrm{Dol}}(X,r)\times \mathbf{S}^2$ into a holomorphic fibration over $\mathbb{P}^1$ that is exactly analytically isomorphic to the generalized  Deligne--Hitchin twistor space $\mathrm{TW}(X,X';r)$, where the complex structures on fibers have been defined as above and the complex structure on the base is the standard one of $\mathbb{P}^1$.

\subsection{Example: Deligne--Hitchin Twistor Spaces}
A typical example is to pick $  X'=\bar X:=(\mathcal{X},-I)$,  the complex conjugate manifold of $X$ which is generally not biholomorphic to $X=(\mathcal{X},I)$ unless there is a real structure on $X$.
 By  gluing $\mathcal{M}_{\mathrm{Hod}}(X,r)$ and
$\mathcal{M}_{\mathrm{Hod}}(\bar X,r)$ with the Deligne isomorphism $\mathbf{d}$, which is explicitly  given by
\begin{align*}
\mathbf{d}: (((E,\bar{\partial}_E),D^\lambda),\lambda)\mapsto (((E,\lambda^{-1}D^\lambda),\lambda^{-1}\bar{\partial}_E),\lambda^{-1}),
\end{align*}
we obtain the usual \emph{Deligne--Hitchin twistor space} $\mathrm{TW}_{\mathrm{DH}}(X,r)$ \cite{CS7,CS8}.

The  deep theorems due to Hitchin and Simpson (for $\lambda=0$) \cite{NH,CS1}, Corlette (for $\lambda=1$) \cite{cor}, and Mochizuki (for $\lambda\neq0$) \cite{TM} exhibit the existence of harmonic metrics on polystable $\lambda$-flat bundles. Given a Hermitian metric $h$ on a $\lambda$-flat bundle $((E,\bar{\partial}_E),D^\lambda)$, one introduces a (0,1)-operator $\delta_h''$ and an (1,0)-operator $\delta_h'$  by the following conditions
 \begin{align*}
  \lambda\partial h(u,v)&=h(D^\lambda u,v)+h(u,\delta_h''v),\\
  \bar{\partial}h(u,v)&=h(\bar{\partial}_E u,v)+h(u,\delta_h'v),
 \end{align*}
then we define  the following four operators
$$
\begin{aligned}
\partial_h&:=\frac{1}{1+|\lambda|^2}\bigg(\bar{\lambda}D^\lambda+\delta_h'\bigg),\ \
\bar{\partial}_h:=\frac{1}{1+|\lambda|^2}\bigg(\bar{\partial}_E+\lambda\delta_h''\bigg),\\
\theta_h&:=\frac{1}{1+|\lambda|^2}\bigg(D^\lambda-\lambda\delta_h'\bigg),\ \
\theta_h^\dagger:=\frac{1}{1+|\lambda|^2}\bigg(\bar{\lambda}\bar{\partial}_E-\delta_h''\bigg),
\end{aligned}$$
which satisfy
$$\begin{aligned}
D^\lambda&=\lambda\partial_h+\theta_h,\ \
\bar{\partial}_E=\bar{\partial}_h+\lambda\theta_h^\dagger,\\
\delta_h'&=\partial_h-\bar{\lambda}\theta_h,\ \
\delta_h''=\bar{\lambda}\bar{\partial}_h-\theta_h^\dagger.
\end{aligned}
$$
It is easy to see that $\mathcal{D}_h=\partial_h+\bar{\partial}_h$ is a unitary connection with respect to the metric $h$, and $\theta_h^\dagger$ is the adjoint of $\theta_h$ in the sense that
$h(\theta_h(u),v)=h(u,\theta_h^\dagger(v))$. When $\lambda\in\mathbb{C}^*$, the harmonicity of the metric $h$ means that  $((E,\bar{\partial}_h),\theta_h)$ is a stable Higgs bundle \cite{TM}.

 We define the  antilinear isomorphism $\sigma:\mathcal{M}_{\mathrm{Hod}}(X,r)\rightarrow \mathcal{M}_{\mathrm{Hod}}(\bar X, r)$  as
 $$
 ((E,\bar{\partial}_E),D^\lambda)\mapsto ((E,\delta_h'),-\delta_h'').
 $$
One easily checks it is  compatible  with the Deligne isomorphism $\mathbf{d}$ in the sense that
$$
\mathbf{d}^{-1}\circ \sigma=\sigma^{-1}\circ\mathbf{d},
$$
 hence we get an antiholomorphic involution $\sigma$ on  $\mathrm{TW}_{\mathrm{DH}}(X,r)$ which covers the antipodal involution $\lambda\mapsto -\bar \lambda^{-1}$ of $\mathbb{P}^1$.
Let $((E,\bar{\partial}_E),\theta)\in \mathcal{M}_{\mathrm{Dol}}(X,r)$ be a stable Higgs bundle with the harmonic metric $h$, then $\{((E,\bar{\partial}_E+\lambda\theta^\dagger_h), \lambda\delta'_h+\theta)\}_{\lambda\in\mathbb{C}}$ defines a holomorphic section of the fibration $\pi:\mathcal{M}_{\mathrm{Hod}}(X,r)\rightarrow \mathbb{C}$. By gluing construction, one obtains a  holomorphic section of the fibration $\pi:\mathrm{TW}_{\mathrm{DH}}(X,r)\rightarrow\mathbb{P}^1$. Such section $s: \mathbb{P}^1\to\mathrm{TW}_{\mathrm{DH}}(X,r)$ described as above  is called the \emph{twistor line} or \emph{preferred section} \cite{CS7,CS8}. These sections are known to have the following nice properties

 \begin{itemize}
\item through each point $((E,\bar{\partial}_E),D^\lambda)\in\mathrm{TW}_{\mathrm{DH}}(X,r)$, there is a unique twistor line;
\item (weight-one property) given a twistor line $s$, its normal bundle $N_s$ is isomorphic to $\mathcal{O}_{\mathbb{P}^1}(1)^{\kappa}$, where $\kappa=\dim_\mathbb{C}\big(\mathcal{M}_{\mathrm{B}}(\mathcal{X},r)\big)=2r^2(g-1)+2$ \cite{HKLR,CS7,CS8};
    \item (reality) twistor line  is $\sigma$-invariant, which can be checked as follows
    \begin{align*}
  \sigma((E,\bar{\partial}_E+\lambda\theta^\dagger_h), \lambda\delta'_h+\theta)=\left\{
                                                                                  \begin{array}{ll}
                                                                                   ((E,\bar{\partial}_E-\bar \lambda^{-1}\theta^\dagger_h), -\bar \lambda^{-1}\delta'_h+\theta), & \hbox{$\lambda\neq 0$,} \\
                                                                                    ((E,\partial_h),\theta_h^\dagger), & \hbox{$\lambda=0$.}
                                                                                  \end{array}
                                                                                \right.
\end{align*}
\end{itemize}

\begin{application}
The following proposition is useful in nonabelian Hodge theory, which can be proved by {\cite[Theorem 5.12]{CS5}} and the separateness of moduli space as what Simpson has done in \cite{CS6}.

\begin{proposition}
Let $\{((E,\bar\partial_i),D^\lambda_i)\}$ be a sequence lying in $\mathcal{M}_\lambda(X,r)$ for some $\lambda\in\mathbb{C}^*$, and assume it has a limit point $((E,\bar\partial_{\infty}),D^\lambda_\infty)$. Denote by $((E,\mathfrak{\bar\partial}_{h_i}),\theta_{h_i})$ and $((E,\bar\partial_{h_\infty}),\theta_{h_\infty})$ the  Higgs bundles corresponding to $((E,\bar\partial_i),D^\lambda_i)$ and $((E,\bar\partial_{\infty}),D^\lambda_\infty)$, respectively. Then the sequence $\{((E,\mathfrak{\bar\partial}_{h_i}),\theta_{h_i})\}$ converge to the point $((E,\bar\partial_{h_\infty}),\theta_{h_\infty})$ in the moduli space $\mathcal{M}_{\mathrm{Dol}}(X,r)$.
\end{proposition}

  Now we revisit it by twistor theory. By assumption, we can assume the sequence $\{((E,\bar\partial_i),D^\lambda_i)\}$ lie in a sufficiently small open subset of $\mathcal{M}_\lambda(X,r)$.  Every point $((E,\bar\partial_i),D^\lambda_i)$    gives rise to a preferred section $s_i$ of the Deligne--Hitchin twistor space. The sequence $\{((E,\mathfrak{\bar\partial}_{h_i}),\theta_{h_i})\}$ has a limit point which is denoted by $(( E,\bar\partial'_\infty),\theta'_\infty)$. Since $N_{s_i}\simeq\mathcal{O}_{\mathbb{P}^1}(1)^{\oplus\kappa}$, we have $H^1(\mathbb{P}^1,N_{s_i})=0$, then  by Kodaira's unobstruction  criterion \cite{K},  there is  a holomorphic section joining $((E,\bar\partial_{\infty}),D^\lambda_\infty)$ to $(( E,\bar\partial'_\infty),\theta'_\infty))$ such that it lies in the closure of a $\sigma$-invariant subset. Therefore, this section is also   $\sigma$-invariant. However, as pointed out by Simpson \cite{CS7}, within a small neighborhood of a given preferred section,  every $\sigma$-invariant holomorphic section is also preferred, thus $( (E,\bar\partial^\prime_\infty),\theta^\prime_\infty)=((E,\bar\partial_{h_\infty}),\theta_{h_\infty})$.\qed
\end{application}

\subsection{Cases of Fixed Determinant}

We also similarly consider the moduli spaces of $\lambda$-flat bundles with fixed determinant and the corresponding twistor spaces. We will employ the following notations.
\begin{itemize}
  \item $ \mathcal{M}_\lambda(X,r,\mathcal{O}_X) $ (with fixed $\lambda\in \mathbb{C}^*$): the moduli space of stable pairs $(E, D^\lambda)$, where $E$ is a holomorphic vector bundle over $X$ of rank $r$ together with the isomorphism $\bigwedge^rE\simeq \mathcal{O}_X$, and $D^\lambda$ is a flat $\lambda$-connection on $E$ such that the operator $\bigwedge^rE\rightarrow \bigwedge^rE\otimes K_X$ coincides with $\lambda\cdot d$.
 \item  $ \mathcal{M}_{\mathrm{dR}}(X,r,\mathcal{O}_X) $: the moduli space $ \mathcal{M}_1(X,r,\mathcal{O}_X) $.
 \item $ \mathcal{M}_{\mathrm{Dol}}(X,r,\mathcal{O}_X)$: the moduli space of stable Higgs bundles $(E,\theta)$ consisting of a holomorphic vector bundle $E$ over $X$ of rank $r$ together with the isomorphism $\bigwedge^rE\simeq \mathcal{O}_X$ and an $\mathcal{O}_X$-linear map $\theta: E\rightarrow E\otimes K_X$ with $\Tr(\theta)=0$.
 \item $\mathcal{N}(X,r,\mathcal{O}_X)$: the moduli space  of stable vector bundles over $X$ of rank $r$ together with the isomorphism $\bigwedge^rE\simeq \mathcal{O}_X$
\item $\mathcal{M}_\mathrm{Hod}(X,r,\mathcal{O}_X)$:  the  moduli space of stable $\lambda$-flat ($\lambda\in \mathbb{C}$) bundles over  $X$  of rank $r$ with the vanishing Chern class and  the fixed determinant $\mathcal{O}_X$.
 \item $\mathcal{M}^s_\lambda(X,r,\mathcal{O}_X)$ (with fixed $\lambda\in \mathbb{C}^*$):  the Zariski dense open subset of $\mathcal{M}_\lambda(X,r,\mathcal{O}_X)$ consisting of $\lambda$-flat bundles such that the underlying vector bundles are stable.
 \item $\mathcal{M}^s_\mathrm{Dol}(X,r,\mathcal{O}_X)$:  the Zariski dense open subset of $\mathcal{M}_\mathrm{Dol}(X,r,\mathcal{O}_X)$ consisting of Higgs bundles such that the underlying vector bundles are stable.
 \item $\mathcal{M}^s_\mathrm{Hod}(X,r,\mathcal{O}_X)$:  the Zariski dense open subset of $\mathcal{M}_\mathrm{Hod}(X,r,\mathcal{O}_X)$ consisting of $\lambda$-flat bundles such that the underlying vector bundles are stable.
 \item $\mathcal{M}_\mathrm{B}(\mathcal{X},\mathrm{SL}(r,\mathbb{C}))$: the moduli space of irreducible representations of the fundamental group $\pi_1(\mathcal{X})$ of $X$ into $\mathrm{SL}(r,\mathbb{C})$.
 \item $\mathrm{TW}(X,X';r,\mathcal{O})$: the generalized Deligne--Hitchin  twistor space for a chosen $X^\prime\in\mathcal{M}({\mathcal{X}})$, which is obtained by gluing $\mathcal{M}_{\mathrm{Hod}}(X,r,\mathcal{O}_X)\times_{\mathbb{C}}\mathbb{C}^*$ and $\mathcal{M}_{\mathrm{Hod}}(X^\prime,r,\mathcal{O}_{X^\prime})\times_{\mathbb{C}}\mathbb{C}^*$ together via the generalized Deligne isomorphism.
\end{itemize}

By fully parallel arguments with those in \cite{BGHL}, we will get the Torelli-type theorem for generalized Deligne--Hitchin twistor spaces.
\begin{theorem}\label{tor}
Let $(X,X^\prime)\in \mathfrak{TD}(\mathcal{X})$ and $(Y,Y^\prime)\in \mathfrak{TD}(\mathcal{Y})$ be  two pairs of Riemann surfaces of genus $g\geq3$. If $\mathrm{TW}(X,X^\prime;r,\mathcal{O})$ is analytically isomorphic to $\mathrm{TW}(Y,Y^\prime;r,\mathcal{O})$ for $r\geq 2$, then either $X\simeq Y,X^\prime\simeq Y^\prime$ or $X\simeq Y^\prime, X^\prime \simeq Y$ in the moduli space $\mathcal{M}(\mathcal{X})$.
\end{theorem}

From the proof of Theorem \ref{23}, we have the following Corollary. 
\begin{corollary}
Let $(X,X^\prime)\in \mathfrak{TD}(\mathcal{X})$ and $(Y,Y^\prime)\in \mathfrak{TD}(\mathcal{Y})$ be  two pairs of Riemann surfaces of genus $g\geq3$. If $\mathrm{TW}(X,X^\prime;r,\mathcal{O})$ is analytically isomorphic to $\mathrm{TW}(Y,Y^\prime;r,\mathcal{O})$ for $r\geq 2$, then the pair $(X,X')$ coincides with $(Y,Y')$ or $(Y',Y)$ in the  space $\mathfrak{TD}(\mathcal{X})$.
\end{corollary}

\section{Twistor Lines and De Rham Sections}
\subsection{Existence of Twistor Lines}
For the generalized Deligne--Hitchin twistor space $\mathrm{TW}(X,X';r)$, we can also introduce the notion of twistor line.

\begin{definition}A holomorphic section $s:\mathbb{P}^1\rightarrow \mathrm{TW}(X,X';r)$ is called a \emph{twistor line} if over $\mathbb{P}^1\backslash\{\infty\}$ it can be written as
\begin{align*}
 s(\lambda)=((E,\bar{\partial}_E+\lambda\theta^\dagger_h), \lambda\delta'_h+\theta), \ \lambda\in \mathbb{P}^1\backslash\{\infty\},
\end{align*}
for a Higgs bundle $((E,\bar{\partial}_E),\theta)\in \mathcal{M}_{\mathrm{Dol}}(X,r)$ with the harmonic metric $h$, in other words,
\begin{align*}
  \mathbf{d}(s(\lambda))=((E,(\bar{\partial}_E+\lambda\theta^\dagger_h+ \delta'_h+\lambda^{-1}\theta)_{X'}^{0,1}),\lambda^{-1}(\bar{\partial}_E+\lambda\theta^\dagger_h+ \delta'_h+\lambda^{-1}\theta)_{X'}^{1,0}), \ \lambda\in \mathbb{P}^1\backslash\{0,\infty\}
\end{align*}
can be holomorphically extended to $\lambda=\infty$.
\end{definition}

Unlike usual Deligne--Hitchin twistor space $\mathrm{TW}_{\mathrm{DH}}(X,r)$, for a general Higgs bundle lying in $\mathcal{M}_{\mathrm{Dol}}(X,r)$, there might be no twistor line through it in generalized Deligne--Hitchin twistor spaces. To show the existence of  twistor lines, we turn to the de Rham sections, which are constructed by means of Simpson's work on  Bialynicki--Birula theory for Hodge moduli spaces \cite{CS9,CW,HH,HP,H}.

\begin{definition} [\cite{CS9}]
Let $E$ be a holomorphic vector bundle over $X$ with a flat connection $\nabla:E\rightarrow E\otimes_{\mathcal{O}_X}K_X$. A decreasing filtration $\mathcal{F}=\{F^p\}_{p=0}^k$ of $E$ by strict subbundles
$$
E=F^0\supset F^1\supset\cdots\supset F^k=0
$$
is called a \emph{Simpson filtration} if it satisfies the following two conditions:
\begin{itemize}
  \item[(1)] Griffiths transversality: $\nabla: F^p\rightarrow F^{p-1}\otimes_{\mathcal{O}_X}K_X$ for $p=1,\cdots,k-1$;
  \item[(2)] graded-semistability: the associated graded Higgs bundle $(\mathrm{Gr}_\mathcal{F}(E),\mathrm{Gr}_\mathcal{F}(\nabla))$, where $\mathrm{Gr}_\mathcal{F}(E)=\bigoplus_{p=0}^{k-1}E^p$ with $E^p=F^p/F^{p+1}$ and $\mathrm{Gr}_\mathcal{F}(\nabla)=\bigoplus_{p=1}^{k-1}\theta^p$ with $\theta^p: E^p\rightarrow E^{p-1}\otimes_{\mathcal{O}_X}K_X$, induced from $\nabla$, is a semistable Higgs bundle.
\end{itemize}
\end{definition}

\begin{theorem} [\cite{CS9}]\footnote{Recently, the similar result for stable Higgs bundles has been established by Hausel and Hitchin (see \cite[Proposition 3.4, Proposition 3.11]{H}).}\label{22}
Let $(E,\nabla)$ be a flat bundle over  $X$.
\begin{enumerate}
\item There exist  Simpson filtrations $\mathcal{F}$ on $(E,\nabla)$.
\item Let $\mathcal{F}_1$, $\mathcal{F}_2$ be two Simpson filtrations on $(E,\nabla)$, then the associated graded Higgs bundles $(\mathrm{Gr}_{\mathcal{F}_1}(E),\mathrm{Gr}_{\mathcal{F}_1}(\nabla))$ and $(\mathrm{Gr}_{\mathcal{F}_2}(E),\mathrm{Gr}_{\mathcal{F}_2}(\nabla))$ are $S$-equivalent.
\item $(\mathrm{Gr}_\mathcal{F}(E),\mathrm{Gr}_\mathcal{F}(\nabla))$ is a stable Higgs bundle iff the Simpson filtration is unique.
\end{enumerate}
\end{theorem}

Consider the natural $\mathbb{C}^*$-action on $\mathbb{M}_{\mathrm{Hod}}(X,r)$ (also on $\mathcal{M}_{\mathrm{Hod}}(X,r)$) via $t\cdot (E,D^\lambda) := (E,tD^\lambda)$, which generalizes the well-known $\mathbb{C}^*$-action on the Dolbeault moduli spaces. Simpson showed that for a fixed $\lambda\neq0$, the limit $\lim\limits_{t\to0}t\cdot(E,D^\lambda)$ exists, and it can be obtained by taking the grading of a Simpson filtration on the associated flat bundle $(E,\lambda^{-1}D^\lambda)$, namely we have the following theorem.

\begin{theorem}[\cite{CS9}]\label{33}
Let $(E,D^\lambda)\in \mathbb{M}_{\mathrm{Hod}}(X,r)$ be a $\lambda$-flat bundle ($\lambda\neq 0$), then we have
$$
\lim\limits_{t\rightarrow 0}t\cdot(E,D^\lambda)=(\mathrm{Gr}_{\mathcal{F}_\lambda}(E),\mathrm{Gr}_{\mathcal{F}_\lambda}(\lambda^{-1}D^\lambda)),
$$
where ${\mathcal{F}_\lambda}$ is a Simpson filtration on the associated flat bundle $(E,\lambda^{-1}D^\lambda)$.
\end{theorem}

Let $\mathrm{TW}(X,X';r)$ be a generalized Deligne--Hitchin twistor space.
Then the $\mathbb{C}^*$-action on $\mathcal{M}_{\mathrm{Hod}}(X,r)$ can be extended to an action on the entire twistor space. Indeed, one defines the $\mathbb{C}^*$-action
on $\mathcal{M}_{\mathrm{Hod}}(X',r)$ by $t\cdot (E,D^\lambda)=(E,t^{-1}D^\lambda)$, then the following diagram commutes
$$\CD
 \mathcal{ M}_{\mathrm{Hod}}(X,r) @>t\cdot>>\mathcal{ M}_{\mathrm{Hod}}(X,r) \\
  @V \mathbf{d} VV @V \mathbf{d} VV  \\
  \mathcal{M}_{\mathrm{Hod}}(X',r) @>t\cdot>> \mathcal{M}_{\mathrm{Hod}}(X',r).
\endCD$$
Fix a parameter $\lambda_0\in \mathbb{C}^*$, then associated to a given $\lambda_0$-flat bundle $((E,\bar{\partial}_E),D^{\lambda_0})\in \mathcal{M}_{\lambda_0}(X,r)$, we have a holomorphic section
 $$
 s_{\lambda_0}(\lambda)=((E,\bar{\partial}_E),\lambda\lambda_0^{-1}D^{\lambda_0})
 $$
 of $\mathcal{M}_{\mathrm{Hod}}(X,r)\times_{\mathbb{C}} \mathbb{C}^*\rightarrow \mathbb{C}^*$, by  Simpson's Theorem \ref{33}, the limit $\lim\limits_{\lambda\rightarrow0}s_{\lambda_0}(\lambda)$ exists, hence
 $s_{\lambda_0}(\lambda)$ can be holomorphically  extended to $\lambda=0$. On the other hand, by the generalized Deligne isomorphism, we have
  $$
  \mathbf{d}(s_{\lambda_0}(\lambda))=((E,(\lambda_0^{-1}D^{\lambda_0}+\bar{\partial}_E)_{X'}^{0,1}),\lambda^{-1}(\lambda_0^{-1}D^{\lambda_0}+\bar{\partial}_E)_{X'}^{1,0}),
  $$
as a section of $\mathcal{M}_{\mathrm{Hod}}(X',r)\times_{\mathbb{C}}\mathbb{C}^*\rightarrow \mathbb{C}^*$.
 Again by Simpson's Theorem \ref{33}, the limit $\lim\limits_{\lambda\rightarrow\infty}\mathbf{d}(s_{\lambda_0}(\lambda))$ also exists, hence $s_{\lambda_0}(\lambda)$ can be holomorphically  extended to $\lambda=\infty$. The extended section is also denoted by $s_{\lambda_0}$, and is  called the \emph{de Rham section} associated to $((E,\bar{\partial}_E),D^{\lambda_0})$. It is worth to mention that, the values $s(0)$ and $s(\infty)$ of the de Rham section $s_{\lambda_0}$ at $0$ and $\infty$ are $\mathbb{C}^*$-fixed points lying in the Dolbeault moduli spaces $\mathcal{M}_{\mathrm{Dol}}(X,r)$ and $\mathcal{M}_{\mathrm{Dol}}(X',r)$, respectively. Therefore, they both have the structures of \emph{systems of Hodge bundles} \cite{CS4}. In this paper, a   $\mathbb{C}^*$-fixed point lying in the Dolbeault moduli space  is called a \emph{ complex variation of Hodge structures} ($\mathbb{C}$-VHS). The subvariety  of $\mathcal{M}_{\mathrm{Dol}}(X,r)$ consisting of stable $\mathbb{C}$-VHSs is denoted by
$\mathcal{V}(X,r)$

\begin{theorem}\label{37}
If $((E,\bar{\partial}_E),\theta)\in \mathcal{M}_{\mathrm{Dol}}(X,r)$ is a $\mathbb{C}$-VHS, then there is a unique twistor line through it in $\mathrm{TW}(X,X';r)$.
\end{theorem}

\begin{proof}The Higgs bundle
$((E,\bar\partial_E),\theta)$  can be  expressed as
$$
\bar\partial_E=
       \begin{pmatrix}
       \bar\partial_{E_1}& & \\
       &\ddots  &   \\
        & &  \bar\partial_{E_k}\\
        \end{pmatrix}, \ \ \theta=
                            \begin{pmatrix}
                            0 & & & \\
                            \theta_1& 0 & &   \\
                            & \ddots &\ddots &\\
                            & & \theta_{k-1}& 0\\
                            \end{pmatrix},
$$
in terms of the holomorphic splitting $E=\bigoplus_{i=1}^kE_k$, which is orthogonal with respect to the harmonic metric $h$.
Let $s:\mathbb{P}^1\rightarrow \mathrm{TW}(X,X^\prime;r)$ be the de Rham section through the flat bundle $((E,\bar\partial_E+\theta^\dagger),\delta'_h+\theta)\in \mathcal{M}_{\mathrm{dR}}(X,r)$. We claim that $s$ provides the desired twistor line.

Over $\mathbb{P}^1\backslash\{0,\infty\}$, we have $s(\lambda)=((E,\bar\partial_E+\theta^\dagger_h),\lambda\delta'_h+\lambda\theta)$.
It is known that \cite{CW,HH0}
 \begin{align*}
   \lim\limits_{\lambda\rightarrow0}s(\lambda)=((E,\bar\partial_E),\theta).
 \end{align*}
  Therefore, it is  necessary to show
\begin{align*}
 s(\lambda)=((E,\bar\partial_E+\lambda\theta^\dagger_h),\lambda\delta'_h+\theta)
\end{align*}
for $\lambda\in \mathbb{P}^1\backslash\{0,\infty\}$, namely there are $C^\infty$-automorphisms $g_\lambda$ such that
\begin{align*}
  g_\lambda^{-1}\bar\partial_E(g_\lambda)+g_\lambda^{-1}\theta^\dagger_hg_\lambda&=\lambda\theta^\dagger_h,\\
  \lambda g_\lambda^{-1}\delta'_h(g_\lambda)+\lambda g_\lambda^{-1}\theta g_\lambda&=\theta.
\end{align*}
One can easily check the above equations have the solution
\begin{align}\label{ga}
 g_\lambda=
     \begin{pmatrix}
     \ \mathrm{Id}_{E_1}& & & \\
     & \ \lambda\mathrm{Id}_{E_2} & &   \\
     &  &\ddots &\\
     & && \lambda^{k-1}\mathrm{Id}_{E_k}\\
     \end{pmatrix}.
\end{align}
We conclude the claim.
\end{proof}

\subsection{Normal Bundles}
For a holomorphic section $s:\mathbb{P}^1\rightarrow \mathrm{TW}(X,X^\prime;r)$, the normal bundle
$N_{s}$ of $s$  is defined  by the following short exact sequence
\begin{align*}
 0\longrightarrow T{\mathbb{P}^1}\longrightarrow s^*T{\mathrm{TW}(X,X';r)}\longrightarrow N_{s}\longrightarrow0.
\end{align*}
Then the fiber of $N_{s}$ at $\lambda\in \mathbb{P}^1$ is $N_s|_{\lambda}\simeq T_{s(\lambda)}\pi^{-1}(\lambda)$.
Since the stability is an open condition, we have 
\begin{align*}
 N_s|_{\lambda}\simeq\left\{
                       \begin{array}{ll}
                        \mathbb{H}^1(\mathbf{HIG}_{s(0)}), & \hbox{$\lambda=0$;} \\
                        \mathbb{H}^1(\mathbf{HIG}_{s(\infty)}), & \hbox{$\lambda=\infty$;} \\
                         \mathbb{H}^1(\mathbf{DR}_{s(\lambda)}), & \hbox{$\lambda\in \mathbb{P}^1\backslash\{0,\infty\}$,}
                       \end{array}
                     \right.
\end{align*}
where $\mathbf{HIG}_{s(0)}$ is the Higgs complex associated to the Higgs bundle $s(0)=((E,\bar\partial_E),\theta)$
\begin{align*}
\mathbf{HIG}_{s(0)}:  E\xlongrightarrow{\theta}E\otimes_{\mathcal{O}_X} K_X,
\end{align*}
$\mathbf{HIG}_{s(\infty)}$ is similar,
and $\mathbf{DR}_{s(\lambda)}$ is the de Rham complex associated to the $\lambda$-flat bundle $s(\lambda)=((E,\bar\partial_E),D^\lambda)$
\begin{align*}
  \mathbf{DR}_{s(\lambda)}: E\xlongrightarrow{D^\lambda}E\otimes_{\mathcal{O}_X} K_X.
\end{align*}
By K\"{a}hler identities, the resolutions of Higgs complex and de Rham complex give rise to the isomorphisms 
\begin{align*}
 \mathbb{H}^1(\mathbf{HIG}_{s(0)})
\simeq& \mathcal{H}^1(((E,\bar\partial_E),\theta))\\
:=&\ \{(\alpha,\beta)\in\mathrm{\Omega}^{0,1}_X(\End(E))\oplus\mathrm{\Omega}^{1,0}_X(\End(E)) :\bar\partial_E\beta+[\theta\wedge \alpha]=\delta'_h\alpha+[\theta^\dagger_h\wedge\beta]=0\}
\end{align*}
for the harmonic metric $h$ on Higgs bundle $((E,\bar\partial_E),\theta)$, 
 and \begin{align*} \mathbb{H}^1(\mathbf{DR}_{s(\lambda)})
\simeq& \mathcal{H}^1(((E,\bar\partial_E),D^\lambda))\\
:=&\ \{(\alpha,\beta)\in\mathrm{\Omega}^{0,1}_X(\End(E))\oplus\mathrm{\Omega}^{1,0}_X(\End(E)) :\bar\partial_E\beta+D^\lambda\alpha=\delta'_h\alpha-\delta''_h\beta=0\}
\end{align*}
for the harmonic metric $h$ on $\lambda$-flat  bundle $((E,\bar\partial_E),D^\lambda)$.

In particular, let $s:\mathbb{P}^1\rightarrow \mathrm{TW}(X,X^\prime;r)$ be a twistor line through a $\mathbb{C}$-VHS $s(0)\in \mathcal{M}_{\mathrm{Dol}}(X,r)$, then we have the local trivializations 
\begin{align*}
 N_s|_{\mathbb{P}^1\backslash\{0\}}&\simeq \mathcal{H}^1(((E,\bar\partial_E),\theta))\times \mathbb{C},\\
 (\alpha+\lambda\beta^\dagger_h,\lambda\alpha^\dagger_h+\beta)&\mapsto ((\alpha,\beta),\lambda)
\end{align*}
and similarly,  $N_s|_{\mathbb{P}^1\backslash\{\infty\}}\simeq \mathcal{H}^1(((E,\bar\partial'_E),\theta'))\times \mathbb{C}$, where $((E,\bar\partial'_E),\theta')=((E,(\bar\partial_E+\delta'_h)_{X'}^{0,1}),(\theta+\theta^\dagger_h)_{X'}^{1,0})$ is the Higgs bundle corresponding to the flat bundle $((E,(\bar{\partial}_E+\theta^\dagger_h+ \delta'_h+\theta)_{X'}^{0,1}),(\bar{\partial}_E+\theta^\dagger_h+ \delta'_h+\theta)_{X'}^{1,0})$ via nonabelian Hodge correspondence.
 
\begin{definition}[\cite{Ver1}] 
Let $s:\mathbb{P}^1\rightarrow \mathrm{TW}(X,X^\prime;r)$ be a holomorphic section, and one writes  $N_{s}\simeq \bigoplus\limits_{i=1}^{\kappa}\mathcal{O}_{\mathbb{P}^1}(a_i)$ with integers $a_i$.
\begin{itemize}
\item $s$ is called a\emph{ balanced rational curve} if $a_1=\cdots=a_\kappa=w$, and $w$ is called the weight of $s$.
  \item $s$ is called a \emph{positive rational curve } if the degree of $N_{s}$, called the degree of $s$, i.e. $\sum\limits_{i=1}^{\kappa}a_i$ is a positive integer.
  \item  The ampleness index of $s$ is defined as
  \begin{align*}
   \mathrm{AI}(s)=\left\{
                    \begin{array}{ll}
                       \#\{a_i:a_i>0\}, & \hbox{$\{a_i:a_i>0\}\neq\emptyset$;} \\
                      0, & \hbox{$\{a_i:a_i>0\}=\emptyset$.}
                    \end{array}
                  \right.
  \end{align*}
If $\mathrm{AI}(s)>0$, namely there is some positive $a_i$,  $s$ is called a \emph{partially ample rational curve}.
  In particular, if $\textrm{AI}(s)=\kappa$, namely  each $a_i$ is a positive integer, $s$ is called an \emph{ample rational curve}.

\end{itemize}
\end{definition}

We have already known that all the twitor lines in Deligne--Hitchin twistor space are balanced rational curves with weight one.   By Kodaira's deformation theory \cite{K}, the infinitesimal deformations of  the section $s$ in $\mathrm{TW}(X,X';r)$ are  given by the elements in the cohomology $H^0(\mathbb{P}^1,N_{s})$, and the corresponding obstruction is described by the cohomology $H^1(\mathbb{P}^1,N_{s})$. Therefore, if $s$ is an ample rational curve, any infinitesimal deformation of  $s$ is unobstruted. The following proposition provides some obstructions of existence of ample rational curves in generalized Deligne--Hitchin twistor space.

\begin{proposition}\label{38}
Assume the generalized Deligne--Hitchin twistor space $\mathrm{TW}(X,X';r)$ admits an ample rational curve $s$.
\begin{enumerate}
\item Let $\delta$ be a global holomorphic $p$-form ($p>0$) on $\mathrm{TW}(X,X^\prime;r)$, and denote the zero locus of $\delta$ by $Z(\delta)$. Then we have $\mathrm{Im}(s)\subseteq Z(\delta)$, hence there is no global holomorphic $p$-form on $\mathrm{TW}(X,X^\prime;r)$ only with isolated zeros. In particular, there do not exist nonzero global holomorphic $p$-forms on the Deligne--Hitchin twistor space $\mathrm{TW}_{\mathrm{DH}}(X,r)$.
  \item Assume $s(0)\in \mathcal{M}_{\mathrm{Dol}}(X,r)$ represents a Higgs bundle with the  nonzero Higgs field. Define $T^*\mathrm{TW}(X,X^\prime;r)(-i)=T^*\mathrm{TW}(X,X^\prime;r)\otimes \pi^*\mathcal{O}_{\mathbb{P}^1}(-i)$ for some positive integer $i$, then $H^1(\mathrm{TW}(X,X^\prime;r),T^*\mathrm{TW}(X,X^\prime;r)(-i))$ is non-zero.
\end{enumerate}
\end{proposition}

\begin{proof}
(1)
 We have the holomorphic splitting
\begin{align*}
  s^*(\bigwedge^pT^*\mathrm{TW}(X,X^\prime;r))\simeq \bigwedge^pN_s^*\oplus (\mathcal{O}_{\mathbb{P}^1}(-2)\otimes\bigwedge^{p-1} N_s^*).
\end{align*}
By assumption, $N^*_{s}\simeq \bigoplus\limits_{i=1}^{\kappa}\mathcal{O}_{\mathbb{P}^1}(-a_i)$ for all $a_i$'s positive, then   $\delta\in H^0(\mathrm{TW}(X,X';r),\bigwedge^pT^*\mathrm{TW}(X,X';r))$ vanishes when restricting on  $s$. In particular, for the Deligne--Hitchin twistor space $\mathrm{TW}_{\mathrm{DH}}(X,r)$, all twistor lines are ample rational curves and  cover the entire space, hence the global holomorphic $p$-form on $\mathrm{TW}_{\mathrm{DH}}(X,r)$ must vanish (also cf. \cite[Proposition 2.2]{NH0}).

 (2) Pick the following  reduced divisor of $\mathrm{TW}(X,X^\prime;r)$
  \begin{align*}
    D=\pi^{-1}(0)+\sum_{\lambda_\mu\in \mathbb{P}^1\backslash\{0\},\mu=1,\cdots,i-1}\pi^{-1}(\lambda_\mu),
  \end{align*}
   then we have the  short exact sequence
 \begin{align*}
  0\rightarrow T^*\mathrm{TW}(X,X^\prime;r)(-i)\rightarrow T^*\mathrm{TW}(X,X^\prime;r)\rightarrow T^*\mathrm{TW}(X,X^\prime;r)|_D\rightarrow0.
 \end{align*}
 If $H^1(\mathrm{TW}(X,X^\prime;r),T^*\mathrm{TW}(X,X^\prime;r)(-i))$ is zero, then the restriction map induces a  surjection
 \begin{align*}
    H^0(\mathrm{TW}(X,X^\prime;r),T^*\mathrm{TW}(X,X^\prime;r)) \twoheadrightarrow H^0(D,T^*\mathrm{TW}(X,X^\prime;r)|_D).
 \end{align*}

On the other hand,  consider the following 1-form $\gamma$ on $\pi^{-1}(0)=\mathcal{M}_{\mathrm{Dol}}(X,r)$  pointwisely defined as
 \begin{align*}
  \gamma|_{u}((\alpha,\beta))=\int_X\Tr(\theta\wedge \alpha),
 \end{align*}
  where $u=((E,\bar\partial_E),\theta)\in \mathcal{M}_{\mathrm{Dol}}(X,r)$, $(\alpha,\beta)\in  \mathcal{H}^1(((E,\bar\partial_E),\theta))$ is a tangent vector at $u$. One easily checks that
  \begin{itemize}
    \item $\gamma$ is well-defined since it is independent of the choice of the representative  of $u$;
    \item $\gamma$ is a $(1,0)$-form;
    \item $\gamma$ is a holomorphic form;
    \item $\gamma$ dose not vanish for some tangent vector $(\alpha,\beta)$ at $u=s(0)$.
  \end{itemize}
  Hence, it follows from (1) that $\gamma$ is not the restriction of some global holomorphic 1-form on $\mathrm{TW}(X,X^\prime;r)$. This contradiction means that $H^1(\mathrm{TW}(X,X^\prime;r),T^*\mathrm{TW}(X,X^\prime;r)(-i))$ is non-zero.
\end{proof}

\begin{theorem}\label{mmmmm}\
\begin{enumerate}
  \item Let $u=((E,\bar\partial_E),\theta)\in \mathcal{M}_{\mathrm{Dol}}(X,r)$ be a $\mathbb{C}$-VHS, and $s$  be a  twistor line through it in $\mathrm{TW}(X,X';r)$, then $s$ is a partially ample rational curve with the ampleness index $\mathrm{AI}(s)\geq g$.
  \item Let $u=((E,\bar\partial_E),0)\in \mathcal{M}_{\mathrm{Dol}}(X,r)$ be a stable vector bundle with zero Higgs field, and $s$  be a  de Rham section through it in $\mathrm{TW}(X,X';r)$, then $s$ is  a partially ample rational curve with the  ampleness index $\mathrm{AI}(s)\geq\jmath= r^2(g-1)+1$. In particular, assume $s$ is a twistor line  through $u$, then $s$ is  an ample rational curve with  degree $2\jmath$ or $3\jmath$, and $s$ is a balanced rational curve with weight one iff $X'\simeq \bar X$ in the Teichm\"{u}ller space  $\mathrm{Tei}(\mathcal{X})$.
\end{enumerate}
\end{theorem}

\begin{proof}
(1) We need to show the normal bundle $N_s$ contains a subbundle $L\simeq\mathcal{O}_{\mathbb{P}^1}(1)^{\oplus\nu}$ for $\nu\geq g$. The subvariety $\mathcal{V}(X,r)$ has many connected  components, namely $\mathcal{V}(X,r)=\coprod\limits_i\mathcal{V}_i(X,r)$, where $i$ is the index of the connected components. Assume $u\in \mathcal{V}_i(X,r)$, and take
 $(\alpha,\beta)\in T_u\mathcal{V}_i(X,r)$,  then one considers the infinitesimal deformation of $s|_{\mathbb{P}^1\backslash\{\infty\}}$ as
 $((E,\bar{\partial}_E+\alpha+\lambda\theta^\dagger_h+\lambda\beta^\dagger_h),\lambda\delta'_h+\lambda\alpha^\dagger_h+\theta+\beta)$ for $\lambda\in \mathbb{C}$.
 By Theorem \ref{37},   we have
\begin{align*}
&\mathbf{d}(((E,\bar{\partial}_E+\alpha+\lambda\theta^\dagger_h+\lambda\beta^\dagger_h),\lambda\delta'_h+\lambda\alpha^\dagger_h+\theta+\beta))\\
 =&\ \mathbf{d}(((E,\bar{\partial}_E+\alpha+\theta^\dagger_h+\beta^\dagger_h),\lambda\delta'_h+\lambda\alpha^\dagger_h+\lambda\theta+\lambda\beta))\\
 =&\ ((E,(\bar{\partial}_E+\alpha+\theta^\dagger_h+\beta^\dagger_h+\delta'_h+\alpha^\dagger_h+\theta+\beta)_{X'}^{0,1}),\lambda^{-1}(\bar{\partial}_E+\alpha+\theta^\dagger_h+\beta^\dagger_h+\delta'_h+\alpha^\dagger_h+\theta+\beta)_{X'}^{1,0})
\end{align*}
for $\lambda\in \mathbb{C}^*$.
It follows that $N_s$ has a subbundle $L$ such that $L|_u\simeq T_u\mathcal{V}_i(X,r)$
and the  transition matrix of $L$   is determined by the identification
\begin{align*}
&(\alpha+\beta^\dagger_h,\lambda\alpha^\dagger_h+\lambda\beta)\\
\simeq&\ (\lambda^{-1}(\alpha+\alpha^\dagger_h)^{1,0}_{X'}+\lambda^{-1}(\beta+\beta^\dagger_h)^{1,0}_{X'},(\alpha+\alpha^\dagger_h)^{0,1}_{X'}+(\beta+\beta^\dagger_h)^{0,1}_{X'}).
\end{align*}
Therefore  $L\simeq\mathcal{O}_{\mathbb{P}^1}(1)^{\oplus\nu}$, where $\nu=\dim_\mathbb{C}T_u\mathcal{V}_i(X,r)$. It is known that  the minimal dimension of the irreducible components of $\mathcal{V}(X,r)$ is exactly $g$ \cite[Theorem 4.4]{HH}.

(2) We need to show the normal bundle $N_s$ contains a subbundle $L\simeq\mathcal{O}_{\mathbb{P}^1}(1)^{\oplus\jmath}$ for $\jmath=r^2(g-1)+1$.Since $(E,\bar\partial_E)$ is a stable vector bundle, we have $\lim\limits_{t\rightarrow 0}t\cdot ((E,\bar\partial_E),D^\lambda)=u\in \mathcal{M}_{\mathrm{Dol}}(X,r)$ for any $\lambda\in \mathbb{C}$. Stability is an open condition, so we consider $(\alpha,\beta)\in \mathcal{H}^1(((E,\bar\partial_E),\nabla))$, and calculate
\begin{align*}
&\mathbf{d}(((E,\bar{\partial}_E+\alpha),\lambda\nabla+\lambda\beta))\\
 =&\ ((E,(\bar{\partial}_E+\alpha+\nabla+\beta)_{X'}^{0,1}),\lambda^{-1}(\bar{\partial}_E+\alpha+\nabla+\beta)_{X'}^{1,0})
\end{align*}
for $\lambda\in \mathbb{C}^*$. Therefore, for any de Rham section $s$ through $u$ in $\mathrm{TW}(X,X';r)$, the normal bundle $N_s$ has a subbundle $L$ such that
 $L \simeq\mathcal{O}_{\mathbb{P}^1}(1)^{\oplus\jmath}$
  and $L|_u\simeq T_u\mathcal{N}(X,r)$, where $u$ is treated as a point lying in the   moduli space $\mathcal{N}(X,r)$ of stable vector bundles over $X$ of rank $r$. Hence  $\jmath=\dim_\mathbb{C}\mathcal{N}(X,r)=r^2(g-1)+1$.

In particular, let $s$ be  a twitor line through $u$. Since $\mathcal{H}^1(((E,\bar\partial_E),0))\simeq\mathfrak{H}^1\times \mathfrak{H}^1$, where
$\mathfrak{H}^1\simeq\{\alpha\in\mathrm{\Omega}^{0,1}_X(\End(E)) :\delta'_h\alpha=0\}\simeq \{\beta\in\mathrm{\Omega}^{1,0}_X(\End(E)) :\bar\partial_E\beta=0\}$, we can consider the deformations of holomorphic structure of bundle and Higgs field separately. The former has been taken into account  in (1), namely $N_s$ has a subbundle $L$ such that $L|_u\simeq T_u\mathcal{N}(X,r)$ and $L\simeq\mathcal{O}_{\mathbb{P}^1}(1)^{\oplus\jmath}$.  For the later, we have
\begin{align*}
&\mathbf{d}(((E,\bar{\partial}_E+\lambda\beta^\dagger_h),\lambda\delta'_h+\beta))\\
 =&\ ((E,(\bar{\partial}_E+\lambda\beta^\dagger_h+\delta'_h+\lambda^{-1}\beta)_{X'}^{0,1}),\lambda^{-1}(\bar{\partial}_E+\lambda\beta^\dagger_h+\delta'_h+\lambda^{-1}\beta)_{X'}^{1,0})
\end{align*}
for $\lambda\in \mathbb{C}^*$, then  $N_s/L\simeq\mathcal{O}_{\mathbb{P}^1}(a)^{\oplus \jmath}$ with $a=1$ or $2$, and $a=1$ iff $X'\simeq \bar X$ in the Teichm\"{u}ller space  $\mathrm{Tei}(\mathcal{X})$.
\end{proof}

\subsection{Energy of de Rham Sections}
In this subsection, we apply the theory developed by Beck, Heller and R\"oser in \cite{BHR1} to investigate the de Rham sections in generalized Deligne--Hitchin twistor spaces. Although their theory is established for usual Deligne--Hitchin twistor spaces, it also works for our  generalized spaces.
Given a flat  bundle $((E,\bar\partial_E),\nabla)\in \mathcal{M}_{\mathrm{dR}}(X,r)$, there is a Simpson filtration $\mathcal{F}=\{F^p\}_{p=0}^k$ on it, then in terms of the  $C^\infty$-spitting $E\simeq \bigoplus_{p=0}^{k-1} E^p$ with $E^p=F^p/F^{p+1}$, we write
\begin{align*}
 \bar\partial_E& =
       \begin{pmatrix}
       \bar\partial_{E^0}&\alpha_{0,1} & \alpha_{0,2}&\cdots&\alpha_{0,k-1}\\
      0 & \bar\partial_{E^1}&\alpha_{1,2} & \cdots&\alpha_{1,k-1}\\
     \vdots &\ddots &\ddots&  \ddots& \vdots \\
       0&\cdots &0 & \bar\partial_{E^{k-2}}&\alpha_{k-2,k-1}\\
       0 &\cdots &0 & 0& \bar\partial_{E^{k-1}}\\
        \end{pmatrix},\\
         \nabla&=
                            \begin{pmatrix}
                            \nabla_{E^0} & \beta_{0,1} & \beta_{0,2}&\cdots&\beta_{0,k-1} \\
                           \theta^1& \nabla_{E^1} &\beta_{1,2} & \cdots&\beta_{1,k-1}  \\
                           \vdots &  \ddots &\ddots & \ddots& \vdots \\
                           0 & \cdots&  \theta^{k-2}&\nabla_{E^{k-2}}&\beta_{k-2,k-1}\\
                          0 & \cdots& 0&  \theta^{k-1}&\nabla_{E^{k-1}}\\
                            \end{pmatrix},
\end{align*}
with induced holomorphic structure $\bar\partial_{E^p}$ and induced connection  $\nabla_{E^p} $ on $E^p$,  morphisms
$\alpha_{p,q}: E^q\rightarrow E^p\otimes_{C^\infty(X)} \Omega^{0,1}_X$, $\beta_{p,q}: E^q\rightarrow E^p\otimes_{C^\infty(X)} \Omega^{1,0}_X$ and nonzero morphisms $\theta^p: E^{p-1}\rightarrow E^{p}\otimes_{C^\infty(X)} \Omega^{1,0}_X$. Let $s:\mathbb{P}^1\rightarrow\mathrm{TW}(X,X^\prime;r)$ be the de Rham section  through $((E,\bar\partial_E),\nabla)$, then the energy of $s$ is defined as (for more details see \cite[Sect.2]{BHR1})
\begin{align*}
  \mathbb{E}(s)=\frac{1}{2\pi\sqrt{-1}}\sum_{p=0}^{k-1}\int_X\Tr(\theta^{p+1}\wedge \alpha_{p,p+1}).
\end{align*}
Actually, via the gauge transformations as the form of \eqref{ga}, one gets   a lift $s(\lambda)=((E,(d''_E)_\lambda),D^\lambda)$ (cf. \cite[Lemma 2.2]{BHR}, or \cite[Lemma 1.10]{BHR1}) over $\mathbb{C}$ of $s$, where
\begin{align*}
 (d''_E)_\lambda&=\begin{pmatrix}
       \bar\partial_{E^0}&\lambda\alpha_{0,1} & \lambda^2\alpha_{0,2}&\cdots&\lambda^{k-1}\alpha_{0,k-1}\\
      0 & \bar\partial_{E^1}&\lambda\alpha_{1,2} & \cdots&\lambda^{k-2}\alpha_{1,k-1}\\
      \vdots &\ddots &\ddots&  \ddots& \vdots \\
      0 &\cdots &0 & \bar\partial_{E^{k-2}}&\lambda\alpha_{k-2,k-1}\\
       0 & \cdots&0 &0 & \bar\partial_{E^{k-1}}\\
        \end{pmatrix}),\\
                      D^\lambda&=      \begin{pmatrix}
                            \lambda\nabla_{E^0} &\lambda^2\beta_{0,1} & \lambda^3\beta_{0,2}&\cdots&\lambda^{k}\beta_{0,k-1}   \\
                           \theta^1& \lambda\nabla_{E^1} &\lambda^2\beta_{1,2} & \cdots&\lambda^{k-1}\beta_{1,k-1}  \\
                           \vdots &  \ddots &\ddots &\ddots& \vdots \\
                           0 & \cdots&  \theta^{k-2}&\lambda\nabla_{E^{k-2}}&\lambda^2\beta_{k-2,k-1}\\
                          0 & \cdots& 0&  \theta^{k-1}&\lambda\nabla_{E^{k-1}}\\
                            \end{pmatrix}).
\end{align*}

\begin{theorem}\label{39}Let $s:\mathbb{P}^1\rightarrow \mathrm{TW}(X,X^\prime;r)$ be a de Rham section, then its energy $\mathbb{E}(s)$ is semi-negative. In particular, $\mathbb{E}(s)=0$ iff the underlying vector bundle of $s(1)$ is stable.
\end{theorem}
\begin{proof} $s(1)$ is a flat bundle $((E,\bar\partial_E),\nabla)$, which carries a Simpson filtration $\mathcal{F}=\{F^p\}_{p=0}^k$, then  taking grading with respect to $\mathcal{F}$, we get a   $\mathbb{C}$-VHS as $s(0)$,  denoted by $((E'=\bigoplus_{p=0}^{k-1} E^p,{\bar\partial_E}'=\bigoplus_{p=0}^{k-1} \bar\partial_{E^p}),\theta=\bigoplus_{p=0}^{k-1} \theta^{p+1})$.
The degree of $(E^0,\bar\partial_{E^0})$ can be calculated by the curvature $F_{\nabla_{E^0}}$ of the connection $\nabla_{E^0}$, then by the flatness of $((E,\bar\partial_E),\nabla)$, we have
\begin{align*}
  \deg(E^0)&=\frac{\sqrt{-1}}{2\pi}\int_X\Tr(F_{\nabla_{E^0}})\\
  &=\frac{\sqrt{-1}}{2\pi}\sum_{p=0}^{k-1}\int_X\Tr(\theta^{p+1}\wedge \alpha_{p,p+1}),
\end{align*}
hence we obtain
\begin{align}\label{zs}
 \mathbb{E}(s)=-\deg(E^0).
\end{align}
Assume $k>1$, then since $((E'/E^0,\bar\partial_{E'}/\bar\partial_{E^0}),\theta/\theta^1)$ is a Higgs subbundle of $((E',\bar\partial_{E'}),\theta)$, the degree of $E'/E^0$ is negative, namely the energy of $s$ is also negative. The energy of $s$ is zero iff $k=1$, i.e. $(E,\bar\partial_E)$ is a stable vector bundle.
\end{proof}

\begin{corollary}Assume $r=2$, then $ \mathbb{E}(s)\in[1-g,0]$.
\end{corollary}
\begin{proof}If $ \mathbb{E}(s)<0$, then there is a nonzero holomorphic morphism $\theta^1:E^0\rightarrow E^1\otimes_{\mathcal{O}_X} K_X$, hence $-2\deg(E^0)+2g-2\geq0$. The conclusion follows.
\end{proof}

\begin{remark}
Define $\tilde s(\lambda)=\lambda^{-1}s(\lambda^2)$ for $\lambda\in \mathbb{C}^*$, which can be holomorphically  extended to $\mathbb{P}^1$.  It is obvious that $\tilde s(\lambda)=s(\lambda)$, which means that, in the language of  \cite[Definition 1.18]{BHR1}, the twist of de Rham section is itself. For the case $r=2$, it follows from \cite[Proposition 2.7]{BHR1} that
\begin{align*}
  \mathbb{E}(s)=2\mathbb{E}(s)-\deg(L),
\end{align*}
for the kernel bundle $L$ of the Higgs field $\theta$ (in this case $L=E^1$), which recovers the equality \eqref{zs}.
\end{remark}

\section{Balanced Metrics}

Let $\widetilde{\mathrm{TW}}(X,X';r)$ be the underlying $C^\infty$-manifold of the generalized Deligne--Hitchin twistor space $\mathrm{TW}(X,X';r)$. Then by our construction, there is  a global $C^\infty$-trivialization $\varphi:\widetilde{\mathrm{TW}}(X,X';r)\simeq\widetilde{\mathcal{M}}_{\mathrm{Dol}}(X,r)\times \mathbf{S}^2$, and let $\varphi_1:=p_1\circ\varphi:\widetilde{\mathrm{TW}}(X,X';r)\rightarrow\widetilde{\mathcal{M}}_{\mathrm{Dol}}(X,r), \varphi_2:=p_2\circ\varphi: \widetilde{\mathrm{TW}}(X,X';r)\rightarrow \mathbf{S}^2$, where $p_i,i=1,2,$ is the natural projection on the $i$-th factor. Introduce a product metric as
$$
G=\varphi_1^*g+\varphi_2^*h,
$$
where $g$ is the  hyperK\"{a}hler metric on $\widetilde{\mathcal{M}}_{\mathrm{Dol}}(X,r)$ \cite{NH,Fuj,bb}, and $h$ is the Fubini--Study metric on $\mathbf{S}^2$. In particular, when $r=1$, the metric $G$ is complete \cite{bb}.   The analytic structure of $\mathrm{TW}(X,X';r)$ endows a complex structure $\mathfrak{I}$ on $\widetilde{\mathrm{TW}}(X,X';r)$, then
we define  a Hermitian form $\omega$ via
$$
\omega(S_1,S_2)=G(\mathfrak{I}S_1,S_2)
$$
for  tangent vectors $S_{1}, S_{2}\in T\widetilde{\mathrm{TW}}(X,X';r)$. To check the   balanceness of $\omega$, one can apply \cite[Proposition 4.5]{KV} (also see {\cite[Theorem 1]{tom}}) since it is just a local problem, then  we immediately obtain the following theorem.

\begin{theorem}\label{balance}
The Hermitian form $\omega$ is balanced, i.e. $d\omega^{\kappa}=0$.
\end{theorem}

 \begin{definition}[\cite{CS1}] 
 Let $\mathcal{E}$ be  a holomorphic vector bundle  equipped with a Hermitian metric $H$  over $\mathrm{TW}(X,X';r)$,
 the  \emph{analytic degree} and the \emph{analytic slope} of $(\mathcal{E},H)$ are defined as
 \begin{align*}
   \deg_\omega(\mathcal{E},H)&=\int_{\mathrm{TW}(X,X';1)}c_1(\mathcal{E},H)\wedge \omega^\kappa,\\
   \mu_\omega(\mathcal{E},H)&=\frac{\deg_\omega(\mathcal{E},H)}{\mathrm{rank}(\mathcal{E})},
 \end{align*}
respectively, where
$$
c_1(\mathcal{E},H)=\frac{\sqrt{-1}}{2\pi}\Tr R(\mathcal{E},H)
$$
for the curvature $(1,1)$-form $R(\mathcal{E},H)$ of the Chern connection associated to the metric $H$. Hence,  the notion of stability can be also defined as usual, namely $(\mathcal{E},H)$ is called $\mu_\omega$-\emph{stable} (resp. $\mu_\omega$-\emph{semistable}) if any  proper coherent subsheaf $\mathcal{E}'\subset\mathcal{E}$  with the induced Hermitian metric (also denoted by $H$)  satisfies $\mu_\omega(\mathcal{E}',H)< \mu_\omega(\mathcal{E},H) $ (resp. $\mu_\omega(\mathcal{E}^\prime,H) \leq\mu_\omega(\mathcal{E},H)$), and $\mu_\omega$-\emph{polystable} if $\mathcal{E}$ is the direct sum of stable bundles with the same slope.
 \end{definition}

\begin{proposition}
Let $\mathcal{E}$ be a  holomorphic vector bundle  $\mathcal{E}$ equipped with a Hermitian metric $H$  over the generalized Deligne--Hitchin twistor space $T\mathrm{TW}(X,X';1)$.
 \begin{enumerate}
 \item If there is a de Rham section  $s:\mathbb{P}^1\rightarrow\mathrm{TW}_{\mathrm{DH}}(X,1)$ such that the pullback bundle $s^*\mathcal{E}$ over $\mathbb{P}^1$ is trivial, then  $(\mathcal{E},H)$ is $\mu_\omega$-semistable.
   \item Let $T_\mathrm{rel}\mathrm{TW}(X,X';1)$ be  the relative holomorphic tangent bundle defined by the short exact sequence
$$
0\rightarrow T_\mathrm{rel}\mathrm{TW}(X,X';1)\rightarrow T\mathrm{TW}(X,X';1)\rightarrow \pi^*T\mathbb{P}^1\rightarrow0.
$$
 If $\mathcal{E}$ is a  proper subbundle  of $ T_\mathrm{rel} \mathrm{TW}(X,X';1)$
then we have 
\begin{align*}
    \mu_\omega(\mathcal{E},G)<\mu_\omega(T \mathrm{TW}(X,X';1),G)<2.
\end{align*}
\end{enumerate}
\end{proposition}

\begin{proof}
(1) When $r = 1$,
we have analytic isomorphisms
 $$
 \pi^{-1}(\lambda)\simeq \left\{
                            \begin{array}{ll}
                             (\mathbb{C}^*)^{2g}, & \hbox{$\lambda\in \mathbb{P}^1\backslash\{0,\infty\}$;} \\
                            \textrm{Jac}(X)\times \mathbb{C}^g  , & \hbox{$\lambda=0$;}\\
                            \textrm{Jac}(X')\times \mathbb{C}^g  , & \hbox{$\lambda=\infty$,}
                            \end{array}
                          \right.
$$
 where  $\textrm{Jac}(\bullet)$ is  the Jacobian variety of $\bullet$, thus
there are isomorphisms among the fibers as the real
analytic groups \cite{sim}. Hence the generalized Deligne-Hitchin twistor space is a $C^\infty$-principal
bundle over $\mathbb{P}^1$. The de Rham section $s$ provides a global section, which leads to a global
$C^\infty$-trivialization of this principal bundle.
It follows that we have an isomorphism
$$
H^2(\mathrm{TW}_{\mathrm{DH}}(X,1),\mathbb{R})\simeq H^2(\mathcal{M}_{\mathrm{Dol}}(X,1),\mathbb{R})\oplus H^2(\mathbb{P}^1,\mathbb{R}).
$$
Since   $s^*c_1(\mathcal{E},H)=c_1(s^*\mathcal{E})=0$ and  $c_1(\mathcal{E},H)|_{\pi^{-1}(\lambda)}=c_1(\mathcal{E}|_{\pi^{-1}(\lambda)},H)$ for any $\lambda\in \mathbb{P}^1$, then  by a result of  Verbitsky ({\cite[Lemma 2.1]{Ver2}}), we find that
  $c_1(\mathcal{E},H)\in  H^2(\mathcal{M}_{\mathrm{Dol}}(X,1),\mathbb{R})$ is primitive, thus  $\deg_\omega(\mathcal{E},H)=0$.
 Therefore we only need to show $\deg_\omega(\mathcal{F},H)\leq0$ for any proper saturated coherent subsheaf $\mathcal{F}\subset\mathcal{E}$. Indeed, by the same argument, we have $\deg_\omega(\mathcal{F},H)=\deg(s^*\mathcal{F})\leq0$, where the last inequality is as a result of the triviality  of $s^*\mathcal{E}$.

(2)
Due to Ricci-flatness of the hyperK\"ahler metric $g$, the holomorphic tangent bundle $(T\pi^{-1}(\lambda),g)$ of $\pi^{-1}(\lambda)$ for $\lambda\in \mathbb{P}^1$ is $\mu_g$-semistable \cite{CS1}.
By a result of Kaledin--Verbitsky ({\cite[Lemma 7.3]{KV}}),  $ (T_\mathrm{rel}\mathrm{TW}(X,X';1),G)$ is also $\mu_\omega$-semistable. Let $s$ be a twistor line through $u=((E,\bar\partial_E),0)\in \mathcal{M}_{\mathrm{Dol}}(X,r)$ in $\mathrm{TW}(X,X';1)$, then by Theorem \ref{mmmmm}-(2),  $\deg(N_s)= \kappa$ or $\frac{3}{2}\kappa$ for $\kappa=2g$, therefore
 \begin{align*}
   \mu(T_\mathrm{rel} \mathrm{TW}(X,X';1)|_s)=\frac{\deg( N_s)}{\kappa}&<\mu( T\mathrm{TW}(X,X';1)|_s)\\
   &=\frac{\deg( N_s)+2}{\kappa+1}\leq \frac{3g+2}{2g+1}<2,
 \end{align*}
which implies that
\begin{align*}
  \mu_\omega(\mathcal{E},G)\leq\mu_\omega(T_\mathrm{rel} \mathrm{TW}(X,X';1),G)<\mu_\omega(T \mathrm{TW}(X,X';1),G)<2
\end{align*}
for any  proper subbundle $\mathcal{E}\subset T_\mathrm{rel} \mathrm{TW}(X,X';1)$.
\end{proof}

\section{Automorphism Groups}

\subsection{Automorphism Groups of Hodge Moduli Spaces}

In this subsection we always assume $g\geq 4$ and $r\geq 2$, and we will generalize the result of Biswas and Heller\cite[Theorem 1.1]{BH}  to the case of higher rank and fixed determinant.

\begin{proposition} \footnote{This Proposition is also obtained in {\cite[Corollary 4.3]{Sin}}.}\label{1}
The Picard group $\Pic(\mathcal{M}_\mathrm{Hod}(X,r,\mathcal{O}_X))$ is isomorphic to $\mathbb{Z}$.
\end{proposition}

\begin{proof}
Let  $\mathring{\mathcal{M}}_\mathrm{Hod}(X,r,\mathcal{O}_X)=\mathcal{M}_\mathrm{Hod}(X,r,\mathcal{O}_X)\backslash \mathcal{M}^s_\mathrm{Hod}(X,r,\mathcal{O}_X)$. It is known that  the codimension of $\mathring{\mathcal{M}}_\mathrm{Hod}(X,r,\mathcal{O}_X)$ in $\mathcal{M}_\mathrm{Hod}(X,r,\mathcal{O}_X)$ is at least 3 \cite{BM,BGL},
  which implies that the pullback morphism provides an isomorphism between Picard groups $\Pic(\mathcal{M}_\mathrm{Hod}(X,r,\mathcal{O}_X))$ and $\Pic(\mathcal{M}^s_\mathrm{Hod}(X,r,\mathcal{O}_X))$.

   Recall that for an algebraic fibration $f:\chi\rightarrow S$ of varieties with geometric connected fibers,  there is an exact sequence
   $$\mathrm{Id}\rightarrow \Pic(S)\rightarrow \Pic(\chi)\rightarrow \Pic(\chi/S)(S)\rightarrow \mathrm{Br}(S),$$
  where $\Pic(\chi/S)(S)$ is the group of sections of the relative Picard scheme $\Pic(\chi/S)$ and $\mathrm{Br}(S)$ denotes the Brauer group of the base $S$. For our case, the fibers of $\pi:\mathcal{M}_\mathrm{Hod}(X,r,\mathcal{O}_X)\rightarrow \mathbb{C}$ are irreducible varieties \cite{CS6}, $\Pic(\mathbb{C})$ and  $\mathrm{Br}(\mathbb{C})$ are both trivial. Hence, there are isomorphisms
  $$
  \Pic(\mathcal{M}_\mathrm{Hod}(X,r,\mathcal{O}_X))\simeq \Pic(\mathcal{M}^s_\mathrm{Hod}(X,r,\mathcal{O}_X))\simeq \Pic(\mathcal{M}^s_\mathrm{Hod}(X,r,\mathcal{O}_X)/\mathbb{C})(\mathbb{C}).
  $$
  Denote the restriction $\pi|_{\mathcal{M}^s_\mathrm{Hod}(X,r,\mathcal{O}_X)}: \mathcal{M}^s_\mathrm{Hod}(X,r,\mathcal{O}_X)\rightarrow\mathbb{C}$ by $\pi_s$.
  Then the fiber $\pi_s^{-1}(0)$ is canonically isomorphic to the cotangent bundle $T^*\mathcal{N}(X,r,\mathcal{O}_X)$ of $\mathcal{N}(X,r,\mathcal{O}_X)$, and the fiber $\pi_s^{-1}(\lambda)$ for any $\lambda\in \mathbb{P}^1\backslash\{0,\infty\}$ is a torsor associated to $T^*\mathcal{N}(X,r,\mathcal{O}_X)$.
  Again by the above exact sequence, we have
  $$\Pic(\pi_s^{-1}(\lambda))\simeq \Pic(\mathcal{N}(X,r,\mathcal{O}_X))$$
  for any $\lambda\in\mathbb{P}^1\backslash\{\infty\}$.
  From a Drezet--Narasimhan's famous result \cite{DN}, which states  that $\Pic(\mathcal{N}(X,r,\mathcal{O}_X))\simeq \mathbb{Z}$, it immediately follows that $\Pic(\mathcal{M}_\mathrm{Hod}(X,r,\mathcal{O}_X))\simeq\mathbb{Z}$.
\end{proof}

\begin{lemma}
Let $\Aut\Big(\mathcal{M}_\mathrm{Hod}(X,r,\mathcal{O}_X),\mathcal{M}^s_\mathrm{Hod}(X,r,\mathcal{O}_X)\Big)$ be the subgroup of $\Aut(\mathcal{M}_\mathrm{Hod}(X,r,\mathcal{O}_X))$ consisting of the automorphisms which preserve both $\mathcal{M}^s_\mathrm{Hod}(X,r,\mathcal{O}_X)$ and $\mathring{\mathcal{M}}_\mathrm{Hod}(X,r,\mathcal{O}_X)$. Then we have an isomorphism of goups
$$
\Aut\Big(\mathcal{M}_\mathrm{Hod}(X,r,\mathcal{O}_X),\mathcal{M}^s_\mathrm{Hod}(X,r,\mathcal{O}_X)\Big)\simeq \Aut(\mathcal{M}^s_\mathrm{Hod}(X,r,\mathcal{O}_X)).
$$
\end{lemma}

\begin{proof}
Since $\mathcal{M}^s_\mathrm{Hod}(X,r,\mathcal{O}_X)$ is a dense open subset of $\mathcal{M}_\mathrm{Hod}(X,r,\mathcal{O}_X)$. We only need to show that the restriction morphism $\Aut(\mathcal{M}_\mathrm{Hod}(X,r,\mathcal{O}_X),\mathcal{M}^s_\mathrm{Hod}(X,r,\mathcal{O}_X))\rightarrow \Aut(\mathcal{M}^s_\mathrm{Hod}(X,r,\mathcal{O}_X))$
 via $\sigma\mapsto \sigma|_Y$ is surjective, namely
 any automorphism of $\mathcal{M}^s_\mathrm{Hod}(X,r,\mathcal{O}_X)$ can be lifted to an automorphism of $\mathcal{M}_\mathrm{Hod}(X,r,\mathcal{O}_X)$.
 Let $\sigma$ be an elements of $\Aut(\mathcal{M}^s_\mathrm{Hod}(X,r,\mathcal{O}_X))$, which can be viewed as a pseudo-automorphism of $\mathcal{M}_\mathrm{Hod}(X,r,\mathcal{O}_X)$. Then $\sigma$ induces an automorphism of $\Pic(\mathcal{M}_\mathrm{Hod}(X,r,\mathcal{O}_X))$, by Proposition \ref{1}, it must be the identity. Choose  $\Lambda$ to be a very ample line bundle on $\mathcal{M}_\mathrm{Hod}(X,r,\mathcal{O}_X)$, since $H^0(\mathcal{M}_\mathrm{Hod}(X,r,\mathcal{O}_X),\Lambda)\simeq H^0(\mathcal{M}^s_\mathrm{Hod}(X,r,\mathcal{O}_X),\Lambda|_{\mathcal{M}^s_\mathrm{Hod}(X,r,\mathcal{O}_X)})$  due to Hartog's  theorem, $\sigma$ induces an isomorphism of $H^0(\mathcal{M}_\mathrm{Hod}(X,r,\mathcal{O}_X),\Lambda)$, hence an automorphism of $\mathbb{P}(H^0(\mathcal{M}_\mathrm{Hod}(X,r,\mathcal{O}_X),\Lambda))$, which gives rise to  an automorphism of $\mathcal{M}_\mathrm{Hod}(X,r,\mathcal{O}_X)$.
 \end{proof}

\begin{theorem}
The group $\Aut(\mathcal{M}_\mathrm{Hod}(X,r,\mathcal{O}_X),\mathcal{M}^s_\mathrm{Hod}(X,r,\mathcal{O}_X))$ satisfies the following short exact sequence
$$
\mathrm{Id}\rightarrow\mathbb{V}\rightarrow\Aut(\mathcal{M}_\mathrm{Hod}(X,r,\mathcal{O}_X),\mathcal{M}^s_\mathrm{Hod}(X,r,\mathcal{O}_X))\rightarrow \Aut(\mathcal{N}(X,r,\mathcal{O}_X))\times \mathbb{C}^*\rightarrow \mathrm{Id},
$$
where the abelian group  $\mathbb{V}$ is  isomorphic to $\mathrm{Mor}(\mathbb{C},H^0(\mathcal{N}(X,r,\mathcal{O}_X),T^*\mathcal{N}(X,r,\mathcal{O}_X)))$ consisting of the algebraic morphisms from $\mathbb{C}$ to the space of holomorphic 1-forms  on $T^*\mathcal{N}(X,r,\mathcal{O}_X)$.
\end{theorem}

\begin{proof}
Let $\sigma$ be an automorphism of $\mathcal{M}_\mathrm{Hod}(X,r,\mathcal{O}_X)$. Denote the following composition
$$
\pi^{-1}(\lambda)\xrightarrow{\sigma}\mathcal{M}_\mathrm{Hod}(X,r,\mathcal{O}_X)\xrightarrow{\pi}\mathbb{C}
$$
by $\tilde \sigma_\lambda:\pi^{-1}(\lambda)\rightarrow \mathbb{C}$.
If $\lambda\in \mathbb{P}^1\backslash\{0,\infty\}$, it has been shown that there are no nonconstant algebraic functions on $ \mathcal{M}_\lambda(X,r,\mathcal{O}_X)$ \cite{CS3,Sin}. Hence, $\tilde \sigma_\lambda$ is a constant map, in other words, $\sigma$ maps the fiber $\pi^{-1}(\lambda)$ to another fiber $\pi^{-1}(\lambda^\prime)$ for some $\lambda'\in \mathbb{C}$. By Riemann--Hilbert correspondence,  $ \mathcal{M}_{\mathrm{dR}}(X,r,\mathcal{O}_X)$ is biholomorphic to the affine variety  $ \mathcal{M}_{\mathrm{B}}(X,\mathrm{SL}(2,\mathbb{C}))$, hence it does not contain any compact submanifold of positive dimension, however, by properness of Hitchin map \cite{NH,hit,CS6},    $ \mathcal{M}_\mathrm{Dol}(X,r,\mathcal{O}_X)$ contains many compact submanifolds of positive dimensions, namely, $\pi^{-1}(\lambda)$
is not biholomorphic to $\pi^{-1}(0)$ for any $\lambda\in \mathbb{P}^1\backslash\{0,\infty\}$.
It follows that  $\lambda'$ actually belongs to $\mathbb{P}^1\backslash\{0,\infty\}$, and $\sigma$ preserves the fiber $\pi^{-1}(0)$.  Therefore, $\sigma$ induces an algebraic automorphism $\tau_\sigma:\mathbb{C}\rightarrow \mathbb{C}$  satisfying \begin{itemize}
\item  $\pi\circ\sigma=\tau_\sigma\circ\pi$,
\item $\tau_\sigma(0)=0$,
\end{itemize}
  namely, $\tau$ is expressed as $\tau_\sigma(z)=\varepsilon_\sigma\cdot z$ for some $\varepsilon_\sigma\in\mathbb{C}^*$.

 We define $\mathcal{M}^{s'}_\mathrm{Hod}(X,r,\mathcal{O}_X)=\bigcup\limits_{\lambda\in \mathbb{P}^1\backslash\{0,\infty\}}\pi^{-1}_s(\lambda)$, it admits a fibration $\varpi: \mathcal{M}^{s'}_\mathrm{Hod}(X,r,\mathcal{O}_X)  \rightarrow\mathcal{N}(X,r,\mathcal{O}_X)$. Treating an  automorphism $\sigma\in \mathrm{Aut}(\mathcal{M}^s_\mathrm{Hod}(X,r,\mathcal{O}_X))$  as an automorphism of $\mathcal{M}^{s'}_\mathrm{Hod}(X,r,\mathcal{O}_X)$, 
there is a map $\eta: \mathbb{C}\rightarrow \Aut(\mathcal{N}(X,r,\mathcal{O}_X))$ such that the following diagram commutes
  $$\CD
   \mathcal{M}^{s'}_\mathrm{Hod}(X,r,\mathcal{O}_X) @>\sigma>>  \mathcal{M}^{s'}_\mathrm{Hod}(X,r,\mathcal{O}_X)\\
    @V \varpi VV @V  \varpi VV  \\
    \mathcal{N}(X,r,\mathcal{O}_X) @>\eta_\sigma>> \mathcal{N}(X,r,\mathcal{O}_X).
  \endCD$$
  Indeed,  pick a point $e\in\mathcal{N}(X,r,\mathcal{O}_X)$, then for any  point $u\in \varpi^{-1}(e)$, the limit $\lim\limits_{c\rightarrow0}c\cdot u$ is exactly $e$. Hence $\sigma$ maps the fiber of $\varpi$ to another fiber since $\sigma$ commutes with the $\mathbb{C^*}$-action on $\mathcal{M}^{s'}_\mathrm{Hod}(X,r,\mathcal{O}_X)$,  namely  the composition
   $$
   \eta(\lambda):=\varpi\circ\wp(\sigma_s(\lambda))\circ \varpi^{-1}: \mathcal{N}(X,r,\mathcal{O}_X)\rightarrow \mathcal{N}(X,r,\mathcal{O}_X)
   $$
   is well-defined.
Consequently, we obtain a group homomorphism
   \begin{align*}
    \mathrm{\Theta}: \Aut(\mathcal{M}^s_\mathrm{Hod}(X,r,\mathcal{O}_X))&\longrightarrow \Aut(\mathcal{N}(X,r,\mathcal{O}_X))\times \mathbb{C}^*\\
    \sigma&\longmapsto (\eta_\sigma,\varepsilon_\sigma).
   \end{align*}

   The automorphism group
   $\Aut(\mathcal{N}(X,r,\mathcal{O}_X))$ has been studied well. More concretely, there are two approaches of producing automorphisms of $\Aut(\mathcal{N}(X,r,\mathcal{O}_X))$ \cite{KP}:
   \begin{itemize}
     \item for $\ell\in \Aut(X)$, and $L\in \Pic^0(X)$ with $L^r\simeq \mathcal{O}_X$, send $E\in \mathcal{N}(X,r,\mathcal{O}_X)$ to $L\otimes \ell^*E$;
     \item for  $\ell\in \Aut(X)$, and $L\in \Pic^0(X)$ with $L^r\simeq \mathcal{O}_X$, send $E\in \mathcal{N}(X,r,\mathcal{O}_X)$ to $L\otimes \ell^*E^\vee$.
   \end{itemize}
   In other words, $\Aut(\mathcal{N}(X,r,\mathcal{O}_X))$ satisfies the following short exact sequence
    $$
   \mathrm{Id}\longrightarrow\mathcal{G}\longrightarrow\Aut(\mathcal{N}(X,r,\mathcal{O}_X))\longrightarrow\mathbb{Z}/2\mathbb{Z}\longrightarrow\mathrm{Id},
   $$
   where the subgroup $\mathcal{G}$ fits in the short exact sequence
   $$
   \mathrm{Id}\longrightarrow \Pic^0(X)_r\longrightarrow\mathcal{G}\longrightarrow\Aut(X)\longrightarrow \mathrm{Id}
   $$
   for $\Pic^0(X)_r$ consisting of the $r$-torsion points in $\Pic^0(X)$.
   
   To show this homomorphism is surjective, we only need to note that there ia a natural $\mathbb{C}^*$-action on $\mathcal{M}^s_\mathrm{Hod}(X,r,\mathcal{O}_X)$, and an elements of $\Aut(\mathcal{N}(X,r,\mathcal{O}_X))$ can be extended to an automorphism of $\mathcal{M}^s_\mathrm{Hod}(X,r,\mathcal{O}_X)$ since there are natural ways to define the tensor product, pullback and dual on $\lambda$-connections.

   Finally, we consider the kernel of the homomorphism $\mathrm{\Theta}$.   Assume $\sigma\in \Ker(\mathrm{\Theta})$, then $\sigma(\pi_s^{-1}(\lambda))=\pi_s^{-1}(\lambda)$ and $\sigma$ preserves the fibers of $\varpi: \mathcal{M}^{s'}_\mathrm{Hod}(X,r,\mathcal{O}_X)  \rightarrow\mathcal{N}(X,r,\mathcal{O}_X)$.  Therefor we have a map $$\rho^\sigma_\lambda:\pi_s^{-1}(\lambda)\rightarrow\mathcal{A}=\bigoplus\limits_{i=1}^rH^0(X,K_X^{\otimes i})$$ defined by the composition
   \begin{align*}
    (E,D^\lambda)&\mapsto (E, \sigma(D^\lambda)-D^\lambda)\mapsto(E, h_E(\sigma(D^\lambda)-D^\lambda)),
   \end{align*}
   where  $(\sigma(D^\lambda)-D^\lambda)\in H^0(X,\mathrm{ad}(E)\otimes K_X)$ for  the subbundle $\mathrm{ad}(E)$ of $\End(E)$ consisting of trace-free endomorphisms of $E$, and $h_E:H^0(X,\mathrm{ad}(E)\otimes K_X)\rightarrow\mathcal{A}$ is the Hitchin map given by
   $\vartheta\mapsto(\Tr(\vartheta^2),\cdots, \Tr(\vartheta^r))$ \cite{hit}.
   When $\lambda\in\mathbb{P}^1\backslash\{0,\infty\}$,  since $\mathrm{codim}_\mathbb{C}(\pi^{-1}(\lambda)\backslash\pi_s^{-1}(\lambda))\geq 3$, there are also no nonconstant algebraic functions on $\pi_s^{-1}(\lambda)$. Hence $\rho^\sigma_\lambda$ is a constant map,
   and then the image is denoted by $\tilde\rho^\sigma_\lambda$. Thus each  $\sigma\in \Ker(\mathrm{\Theta})$ gives rise to  a morphism
   \begin{align*}
   \rho: \mathbb{C}^*&\longrightarrow\mathcal{A}\\
   \lambda&\longmapsto\tilde\rho^\sigma_\lambda,
   \end{align*}
   which can be extended to the entire $\mathbb{C}$. Moreover, since $\dim_\mathbb{C}H^0(X,\mathrm{ad}(E)\otimes K_X)=\dim_\mathbb{C}\mathcal{A}=(r^2-1)(g-1)$ \cite{hit}, the map $(E,D^\lambda)\mapsto(E, \sigma(D^\lambda)-D^\lambda)$ is independent of $D^\lambda$, hence yields a section $\tilde\sigma_\lambda$ of $T^*\mathcal{N}(X,r,\mathcal{O}_X)$ by Serre duality. Consequently, we obtain a morphism
   \begin{align*}
 \varrho_\sigma: \mathbb{C}^*&\longrightarrow H^0(\mathcal{N}(X,r,\mathcal{O}_X),T^*\mathcal{N}(X,r,\mathcal{O}_X))\\
   \lambda&\longmapsto\tilde\sigma_\lambda,
   \end{align*}
   which provides an isomorphism $\Ker(\mathrm{\Theta})\simeq \mathrm{Mor}(\mathbb{C},H^0(\mathcal{N}(X,r,\mathcal{O}_X),T^*\mathcal{N}(X,r,\mathcal{O}_X))$ via $\sigma\mapsto\varrho_\sigma$. We complete the proof.

\end{proof}

\subsection{Automorphism Groups of Generalized Deligne--Hitchin Twistor Spaces}

In \cite{BHR}, Biswas, Heller and R\"oser proved  that  the automorphism of the Deligne--Hitchin twistor space which is homotopic to the identity automorphism maps  fibers of the fibration $\mathrm{TW}_{\mathrm{DH}}(X,r)\to\mathbb{P}^1$ to fibers. In the following, we will generalize this property to the generalized Deligne--Hitchin twistor space $\mathrm{TW}(X,X';r)$.

\begin{theorem}\label{auto}
Let $\Aut_0(\mathrm{TW}(X,X^\prime;r))$ be the identity component of the holomorphic automorphism  group $\Aut(\mathrm{TW}(X,X';r))$ of $\mathrm{TW}(X,X';r)$, then it satisfies the following short exact sequence
$$
\mathrm{Id}\longrightarrow\mathfrak{K}\longrightarrow\Aut_0(\mathrm{TW}(X,X';r))\longrightarrow\mathbb{C}^*\longrightarrow\mathrm{Id},
$$
where each element of $\mathfrak{K}$ preserves the fiber of the fibration $\pi:\mathrm{TW}(X,X';r)\rightarrow\mathbb{P}^1$.
\end{theorem}

\begin{proof}The approach we use here is a slight modification of that in \cite[Theorem 5.1]{BHR}.
Let $\sigma\in \Aut_0(\mathrm{TW}(X,X';r))$. Firstly, we show that $\sigma$ maps  fibers of $\pi:\mathrm{TW}(X,X';r)\rightarrow\mathbb{P}^1$ to  fibers. Otherwise, we assume $\sigma$ does not map the fiber $\pi^{-1}(\lambda_0)$ to a fiber for some $\lambda_0\in\mathbb{P}^1$. We can find $u_1=((E,\bar\partial_E),D^{\lambda_1})\in \sigma(\pi^{-1}(\lambda_0))\bigcap \pi^{-1}(\lambda_1)$ with $\lambda_1\in \mathbb{P}^1\{0,\lambda_0,\infty\}$ such that
\begin{itemize}
  \item there is a twistor line $s_{u_1}$ through $u_1$ with $s_{u_1}(0)$ being a $\mathbb{C}$-VHS,
  \item there is a small analytic open neighborhood $U\subset \mathbb{P}^1\backslash\{0,\lambda_0,\infty\}$ of $\lambda_1$ such $U\subset \pi(\sigma(\pi^{-1}(\lambda_0)))$,
      \item the tangent map $d\pi|_{u_1}: T_{u_1}\sigma(\pi^{-1}(\lambda_0))\rightarrow T_{\lambda_1}\mathbb{P}^1$ of the projection $\pi$ at $u_1$ is surjective,
          \item there exists  $(\alpha,\beta)\in T_{u_1}\sigma(\pi^{-1}(\lambda_0))$ with nonzero $\alpha\in\mathrm{\Omega}^{0,1}_X(\End(E))$ and $\beta\in\mathrm{\Omega}^{1,0}_X(\End(E))$ such that $p_v(\alpha,\beta)$ lies in  $N_{s_{u_1}}|_{\lambda_2}\simeq T_{u_1^{\prime}}\mathcal{M}_{\mathrm{dR}}(X,r)$ with some $\lambda_2\neq \lambda_1\in U$, where
$p_v: T_{u_1}\sigma(\pi^{-1}(\lambda_0))\rightarrow T_{u_1^{\prime}}\mathcal{M}_{\mathrm{dR}}(X,r)
$ is the projection under the isomorphism $\mathrm{TW}(X,X^\prime;r)|_{\mathbb{P}^1\backslash\{0,\infty\}}\simeq \mathcal{M}_{\mathrm{dR}}(X,r)\times\mathbb{C}^*$ for $u_1^\prime=((E,\bar\partial_E),\lambda_1^{-1}D^{\lambda_1})$.
\end{itemize}

By Theorem \ref{mmmmm}-(1), we can pick two points $(\alpha,\beta)$ lying in $ T_{u_1}\sigma(\pi^{-1}(\lambda_0)$ to determine a section $s_{u_1}^\prime$ of subbundle $L\subset N_{s_{u_1}}$ which is semi-stable and of positive degree such that   $s_{u_1}^\prime$ produces  a section $\widetilde{s_{u_1}^\prime}: \mathbb{P}^1\rightarrow\mathrm{TW}(X,X^\prime;r)$ satisfying that  $\sigma^{-1}(\widetilde{s_{u_1}^\prime})$ intersects the fiber $\pi^{-1}(\lambda_0)$ at least twice. This contradicts  the fact that the element of $\Aut_0(\mathrm{TW}(X,X^\prime;r))$ maps a section to another section.

As a consequence, the holomorphic automorphism $\sigma\in \Aut_0(\mathrm{TW}(X,X^\prime;r))$ induces a holomorphic automorphism $\tau_\sigma\in \Aut(\mathbb{P}^1)$, which has a form $\tau_\sigma(\lambda)=\frac{a\lambda+b}{c\lambda+d}$ with
$\left(
           \begin{array}{cc}
                 a & b\\
                 c & d \\
           \end{array}
       \right)\in \mathrm{GL}(2,\mathbb{C})$.
       Since $\sigma$ must map the fiber $\pi^{-1}(0)$ to itself or to the fiber $\pi^{-1}(\infty)$, we have
$$
\tau_\sigma(\lambda)=\left\{
    \begin{aligned}
    \varepsilon_\sigma\cdot \lambda, \qquad& \tau_\sigma(0)=0,\  \tau_\sigma(\infty)=\infty; \\
                     \frac{\lambda}{\varepsilon_\sigma },\qquad & \tau_\sigma(0)=\infty,\  \tau_\sigma(\infty)=0,
    \end{aligned}
                 \right.
$$
for some $\varepsilon_\sigma\in\mathbb{C}^*$. The theorem follows.
\end{proof}


\begin{thebibliography}{9}

 \bibitem{ag} 
 Abbes, A.,  Gros, M.,  Tsuji, T.: The $p$-adic Simpson correspondence. Annals of Mathematics Studies, Princeton University Press (2016)


 \bibitem{BHR1}
 Beck, F., Heller, S.,  R\"{o}ser, M.: Energy of sections of the Deligne–Hitchin twistor space. Math. Ann. \textbf{380}, 1169-1214 (2021)

 \bibitem{bb} 
 Biquard, O., Boalch, P.:  Wild non-abelian Hodge theory on curves. Compo. Math. \textbf{140}, 179-204 (2004)
 
 \bibitem{BGHL} 
 Biswas, I., G\'{o}mez, T., Hoffman, N., Logares,  M.: Torelli theorem for the Deligne--Hitchin moduli space. Comm. Math. Phys. \textbf{290}, 357-369 (2009)

 \bibitem{BGH} 
 Biswas, I., G\'{o}mez, T., Hoffman, N.: Torelli theorem for the Deligne--Hitchin moduli space, II. Doc. Math. \textbf{18}, 1177-1189 (2012)

 \bibitem{BGL} 
 Biswas, I., Gothen, P., Logares, M.: On moduli spaces of Hitchin pairs. Math. Proc. Camb. Phil. Soc. \textbf{151}, 441-457 (2011)

 \bibitem{BH} 
 Biswas, I., Heller, S.: On the automorphisms of a rank one Deligne--Hitchin moduli space. SIGMA Symmetry, Integrability and Geometry: Methods and Applications, vol. 13, Paper No. 072 (2017)

 \bibitem{BHR} 
 Biswas, I., Heller, S.,  R\"{o}ser, M.: Real holomorphic sections of the Deligne--Hitchin twistor space. Comm. Math. Phys. \textbf{366}, 1099-1133 (2019)

 \bibitem{BM} 
 Biswas, I., Mu\~noz, V.: Torrelli theorem for moduli spaces of $\mathrm{SL}(r,\mathbb{C})$-connections on a compact Riemann surface. Comm. Contemp. Math. \textbf{11}, 1-26 (2009)


 \bibitem{CW} 
 Collier, C., Wentworth, R.: Conformal limits and the Bialynicki-Birula stratification of the space of $\lambda$-connections.  Adv. Math. \textbf{350}, 1193-1225 (2019)

 \bibitem{cor} 
 Corlette, K.: Flat $G$-bundles with canonical metrics. Jour. Diff. Geom. \textbf{28}, 361-382 (1988)

 \bibitem {D}  
 Deligne, P.: Various letters to C. Simpson.

 \bibitem{DN} 
 Drezet, J., Narasimhan, M.S.: Groupe de Picard des vari\'et\'es de modules de fibr\'es semi-stables sur les courbes alg\'ebriques. Invent. Math. \textbf{97}, 53-94 (1989)
 
 \bibitem{FM2012}
 Farb, B., Margalit, D.: A primer on mapping class groups. Princeton Mathematical Series \textbf{49}, Princeton University Press (2012) 

 \bibitem{F}  
 Fargues, L.: From local class field to the curve and vice versa. \href{https://hal.archives-ouvertes.fr/hal-01216763v2}{HAL Id: hal-01216763}

 \bibitem{Fuj} 
 Fujiki, A.: HyperK\"ahler structures on the moduli space of flat bundles. Prospects in complex geometry. In: Lecture Notes in Mathematics,  vol. 1468, pp. 1-83. Springer (1991)


 \bibitem{H}
 Hausel, T., Hitchin, N.: Very stable Higgs bundles, equivariant multiplicity and mirror symmetry. Invent. Math. \textbf{228}, 893-989 (2022)

 \bibitem{NH0} 
 Hitchin, N.: K\"ahlerian twistor spaces. Proc. London Math. Soc. \textbf{3}, 133-150 (1981)

 \bibitem{NH} 
 Hitchin, N.: The self-duality equations on a Riemann surface. Proc. London Math.  Soc. \textbf{1},  59-126 (1987)

 \bibitem{hit} 
 Hitchin, N.: Stable bundles and integrable systems. Duke Math. J. \textbf{54}, 91-114 (1987)

 \bibitem{HKLR} 
 Hitchin, N., Karlhede, A., Lindstr\"{o}m, U., Ro\v{c}ek, M.: Hyper-K\"{a}hler metrics and supersymmetry. Comm. Math. Phys. \textbf{108}, 535-589 (1987)

 \bibitem{HH0} 
 Hu, Z., Huang, P.: Simpson--Mochizuki correspondence for $\lambda$-flat bundles.  J. Math. Pures Appl. \textbf{164}, 57-92 (2022)
 
 \bibitem{HH} 
 Hu, Z., Huang, P.: Simpson filtration and oper stratum conjecture. manus. math. \textbf{167}, 653-673 (2022) 
 
 \bibitem{HP} 
 Huang, P.: Non-Abelian Hodge theory and related topics. SIGMA \textbf{16}, 029 (2020)
 
 \bibitem{KV}
 Kaledin, D., Verbitsky, M.: Non-Hermitian Yang--Mills connections. Selecta Math. \textbf{4}, 279-320 (1998)
 
 \bibitem{K} 
 Kodaira, K.:  A theorem of completeness of characteristic systems for analytic families of compact submanifolds of complex manifolds. Ann. Math. \textbf{84}, 146-162 (1962)
 
 \bibitem{KP} 
 Kovidakis, A., Pantev, T.: The automorphsim group of the moduli space of semistable vector bundles. Math. Ann. \textbf{302}, 225-268 (1995)
 
 
 \bibitem{TM} 
 Mochizuki, T.: Kobayashi-Hitchin correspondence for tame harmonic bundles II. Geom. \& Topol. \textbf{13}, 359-455 (2009)
 
 
 \bibitem{Pen} 
 Penrose, R.: Non-linear gravitons and curved twistor theory. Gen. Relativ. Gravitation \textbf{7}, 31 (1976)
 
 \bibitem{ps} 
 Scholze, P.: \emph{p}-Adic geometry.  Proceedings of the International Congress of Mathematicians (ICM 2018), pp. 899-933 (2019)
 
 \bibitem{CS1} 
 Simpson, C.: Constructing of variations of Hodge structure using Yang-Mills theory and applications to uniformization.  J. Amer. Math. Soc. \textbf{1}, 867-918 (1988)
 
 \bibitem{CS3} 
 Simpson, C.: A lower bound for the size of monodromy of systems of ordinary differential equations. In: Algebraic geometry and analytic geometry, ICM-90 Satell. Conf. Proc., Springer, Tokyo,  198-230 (1991)
 
 \bibitem{CS4} 
 Simpson, C.: Higgs bundles and local systems. Inst. Hautes \'Etudes Sci. Publ. Math. \textbf{75}, 5-95 (1992)
 
 \bibitem{sim} 
 Simpson, C.: Subspaces of moduli spaces of rank one local systems. Ann. Sci. \'{E}cole Norm. Sup. \textbf{26}, 361-401 (1993)
 
 \bibitem{CS5} 
 Simpson, C.: Moduli of representations of the fundamental group of a smooth projective variety I. Inst. Hautes \'Etudes Sci. Publ. Math. \textbf{79}, 47-129 (1994)
 
 \bibitem{CS6} 
 Simpson, C.: Moduli of representations of the fundamental group of a smooth projective variety II. Inst. Hautes \'Etudes Sci. Publ. Math. \textbf{80}, 5-79 (1994)
 
 \bibitem{CS7} 
 Simpson, C.: The Hodge filtration on nonabelian cohomology. In: Algebraic geometry-Santa Cruz 1995. Proc. Sympos. Pure Math. \textbf{62}, Part 2, 217-281. Amer. Math. Soc., Providence, RI (1997)
 
 \bibitem{CS8} 
 Simpson, C.: A weight two phenomenon for the moduli of rank one local systems on open varieties. In: From Hodge theory to integrability and TQFT $tt^*$-geometry. Proc. Sympos. Pure Math. \textbf{78}, 175-214. Amer. Math. Soc., Providence, RI (2008)
 
 \bibitem{CS9} 
 Simpson, C.: Iterated destabilizing modifications for vector bundles with connection. In: Vector bundles and complex geometry. Contemp. Math. \textbf{522}, 183-206. Amer. Math. Soc., Providence, RI (2010)

 \bibitem{CS10}
 Simpson, C.: The twistor geometry of parabolic structures in rank two. Proc. Indian Acad. Sci. Math. Sci. \textbf{132}, paper no. 54 (2022) 

 \bibitem{CS11}
 Simpson, C.: Twistor space for local systems on an open curve. \href{https://arxiv.org/abs/2303.13947}{arXiv:2303.13947}

 \bibitem{Sin}
 Singh, A.: On the moduli space of $\lambda$-connections. Proc. Amer. Math. Soc. \textbf{149}, 459-470 (2021) 
 
 \bibitem{tom} 
 Tomberg, A.: Twistor spaces of hypercomplex manifolds are balanced. Adv. Math. \textbf{280},  282-300 (2015)
 
 
 \bibitem{Ver1} 
 Verbitsky, M.: Holography principal for twistor spaces. Pure Appl. Math. Q. \textbf{10}, 325-354 (2014)
 
 \bibitem{Ver2} 
 Verbitsky, M.: Hyperholomorphic  bundles  over  a  hyperk\"{a}hler  manifold. Jour.  Alg. Geom. \textbf{5}, 633-669 (1996)
 
 
 
\end{thebibliography}
 \end{document}